\documentclass[10pt,oneside,leqno]{amsart}
\usepackage{amsxtra}
\usepackage{amsopn}
\usepackage{color}
\usepackage{amsmath,amsthm,amssymb}
\usepackage{amscd}
\usepackage{amsfonts}
\usepackage{latexsym}
\usepackage{verbatim}

\theoremstyle{plain}
\newtheorem{theorem}{Theorem}[section]

\newtheorem{lemma}[theorem]{Lemma}
\newtheorem{proposition}[theorem]{Proposition}

\newtheorem{remark}[theorem]{Remark}
\newtheorem{example}[theorem]{Example}

\newtheorem{remark-question}[section]{Remark-Question}

\newcommand\C{{\mathbb C}}

\newcommand\R{{\mathbb R}}
\newcommand\Z{{\mathbb Z}}

\newcommand\trace{{\rm tr}}

\newcommand\frb{{\mathfrak b}}
\newcommand\frc{{\mathfrak c}}
\newcommand\fre{{\mathfrak e}}
\newcommand\frg{{\mathfrak g}}
\newcommand\frh{{\mathfrak h}}

\newcommand\Real{{\mathfrak R}{\frak e}\,} 
\newcommand\Imag{{\mathfrak I}{\frak m}\,}

\newcommand\db{{\bar{\partial}}}

\definecolor{fondo}{rgb}{0.93,0.93,0.93}

\begin{document}
\title[]{Six dimensional solvmanifolds with holomorphically trivial canonical bundle}
\subjclass[2000]{
}

\author{Anna Fino}
\address[A. Fino]{Dipartimento di Matematica\\
 Universit\`a di Torino\\
Via Carlo Alberto 10\\
10123 Torino, Italy} \email{annamaria.fino@unito.it}

\author{Antonio Otal}
\address[A. Otal]{Centro Universitario de la Defensa\,-\,I.U.M.A.\\
Academia General Militar\\
Ctra. de Huesca s/n. 50090 Zaragoza\\
Spain}
\email{aotal@unizar.es}

\author{Luis Ugarte}
\address[L. Ugarte]{Departamento de Matem\'aticas\,-\,I.U.M.A.\\
Universidad de Zaragoza\\
Campus Plaza San Francisco\\
50009 Zaragoza, Spain}
\email{ugarte@unizar.es}

\maketitle

\begin{abstract}
We determine the 6-dimensional solvmanifolds admitting an invariant complex structure with holomorphically trivial canonical bundle.
Such complex structures are classified up to isomorphism, and the existence of
strong K\"ahler with torsion (SKT),  generalized Gauduchon, balanced and strongly Gauduchon metrics is studied.
As an application we construct a holomorphic family $(M,J_a)$ of compact complex manifolds
such that $(M,J_a)$ satisfies the $\partial\db$-lemma and admits a balanced metric
for any $a\not=0$, but the central limit neither satisfies the $\partial\db$-lemma nor admits balanced metrics.
\end{abstract}


\vskip.6cm

\hskip7.8cm {\it In memory of Sergio Console}

\vskip.6cm

\section{Introduction}

\noindent Any compact complex surface with holomorphically  trivial canonical bundle is isomorphic to a K3 surface,
a torus, or a Kodaira surface; the first two are K\"ahler, and the latter is a
nilmanifold, i.e. a compact quotient of a nilpotent Lie group by a lattice.
It is well known that in any real dimension $2n$ the
canonical bundle of  a  nilmanifold  $\Gamma\backslash G$  endowed with an invariant complex structure
is holomorphically trivial, where by invariant complex structure
we mean one induced by a complex structure on the Lie algebra of the nilpotent Lie group $G$.
In fact, Salamon proved in \cite{S} the existence of a closed non-zero invariant $(n,0)$-form
(see also~\cite{BDV} for some applications of this fact to hypercomplex nilmanifolds). For 6-dimensional nilmanifolds the classification
of invariant complex structures $J$ as well as the existence of some special Hermitian metrics
(as SKT, generalized first Gauduchon, balanced and strongly Gauduchon metrics) with respect to such $J$'s have been studied in \cite{COUV,FPS,FU}.
In this paper we are interested in the Hermitian geometry of 6-dimensional solvmanifolds admitting an
invariant complex structure with holomorphically trivial canonical bundle, i.e.
with an \emph{invariant} non-zero closed (3,0)-form (see Proposition~\ref{holglobalform} for details).
Throughout the paper, by a solvmanifold we mean a compact quotient $\Gamma\backslash G$ of a \emph{non-nilpotent} solvable
Lie group $G$ by a lattice~$\Gamma$.

Since the complex structures we consider are invariant, they come from complex structures on the Lie algebra $\frg$ of $G$.
Given an almost complex structure $J\colon\frg\longrightarrow\frg$ on a 6-dimensional Lie algebra $\frg$, the existence
of a non-zero closed form of bidegree (3,0)
implies that  $J$ is integrable, i.e. that $J$ is a complex structure.
If the Lie algebra is nilpotent then both conditions are equivalent~\cite{S}.
The 6-dimensional nilpotent Lie algebras admitting a complex structure
were classified in \cite{S} and recently these complex structures
have been classified up to equivalence~\cite{COUV}.
For solvable Lie algebras admitting a complex structure
there exists a general classification  only in real dimension four \cite{Ovando},
but no  general result is known in higher dimension.

In this paper we consider the bigger
class of 6-dimensional solvable Lie algebras $\frg$ and look for almost complex structures admitting
a non-zero closed (3,0)-form. Notice that the latter condition implies that $b_3(\frg)\geq 2$.
We will consider both the indecomposable and the decomposable solvable Lie algebras which are unimodular
as we aim to find invariant complex structures  with non-zero closed (3,0)-form on solvmanifolds \cite{Milnor}.
We recall that the 6-dimensional indecomposable solvable Lie algebras
have been classified by Turkowski \cite{T}, Mubarakzyanov \cite{Mu} and Shabanskaya,
who recovered and corrected in \cite{Sha} the work of Mubarakzyanov
(see the Appendix for more details).

In Section~\ref{sec-preli} we consider the formalism of stable forms \cite{H}, together with ideas in
\cite{C,F-Schu1,F-Schu2,Schu}, and we explain in detail the method that we follow to classify unimodular
solvable Lie algebras admitting a complex structure with non-zero closed (3,0)-form. The classification is obtained
in Propositions~\ref{decomposable} and~\ref{indecomposable} for the decomposable and indecomposable cases,
respectively. As a consequence we have that
if a solvmanifold admits a complex structure arising from an invariant non-vanishing holomorphic (3,0)-form, then its
underlying Lie algebra $\frg$ must be isomorphic to one in the list of Theorem~\ref{main-th}, i.e.
$\frg\cong \frg_1,\frg_2^{\alpha} \ (\alpha\geq0),\frg_3,\ldots,\frg_8$ or $\frg_9$.
The Lie algebras $\frg_1,\frg_2^{\alpha} \ (\alpha\geq0)$ and $\frg_3$ are decomposable, whereas $\frg_4,\ldots,\frg_8$
and $\frg_9$ are indecomposable.
The Lie algebra $\frg_8$ is precisely the real Lie algebra underlying Nakamura manifold \cite{Nak}.
Moreover, using some results in \cite{Bock}, we ensure in Proposition~\ref{existence-of-lattices}
the existence of a lattice for the simply-connected Lie groups
associated to the Lie algebras in the list, although for $\frg_2^{\alpha}$ we are able to find a lattice
only for a countable number of different values of $\alpha$ (note that one cannot expect a lattice to exist for
any real $\alpha>0$ according to \cite[Prop. 8.7]{Wi}).

In Section~\ref{sec-reduc} we consider the whole space of complex structures having closed (3,0)-form
and classify them up to equivalence
(see Theorem~\ref{solvmanifolds-complex-classification}).
It turns out that the classification is finite, except for $\frg_3$ (Proposition~\ref{complex-moduli-g3})
and $\frg_8$ (Proposition~\ref{complex-moduli-g8}).
As a consequence of this complex classification, we study in Section~\ref{sec-metrics} the existence of several special
Hermitian metrics on the corresponding complex solvmanifolds. The SKT geometry is studied in Theorem~\ref{skt-geometry}
and provides a new example of compact SKT manifold based on the Lie algebra $\frg_4$.
In Theorem~\ref{gamma1-geometry} we investigate the existence of \emph{invariant} first Gauduchon metrics, in the sense of \cite{FWW},
and find that a solvmanifold corresponding to $\frg_6$
has invariant first Gauduchon metrics which are not SKT, moreover it does not admit any SKT metric
(notice that any invariant first Gauduchon metric on a 6-nilmanifold
is necessarily SKT as proved in \cite[Proposition 3.3]{FU}).
Finally, Theorems~\ref{balanced-geometry} and~\ref{sG-geometry} deal with the existence of balanced metrics~\cite{Mi}
and
strongly Gauduchon metrics~\cite{Pop0,Pop2,Pop09},
respectively, and provide new compact examples
of such Hermitian geometries.

Our goal in Section~\ref{hol-deform} is to construct a holomorphic family of compact complex manifolds
$(M,J_a)_{a\in \Delta}$, $\Delta=\{a\in \mathbb{C} \mid |a|<1\}$, such that $(M,J_a)$ satisfies the $\partial\db$-lemma
and admits balanced metric
for any $a\not=0$, but $(M,J_0)$ neither satisfies the $\partial\db$-lemma nor admits balanced metric.
The construction is based on the balanced Hermitian geometry studied in Theorem~\ref{balanced-geometry} for $\frg_8$,
the Lie algebra underlying Nakamura manifold, together with a recent result by Angella and Kasuya \cite{AK}
on deformations of the holomorphically parallelizable Nakamura manifold.
Notice that the central limit $(M,J_0)$ admits strongly Gauduchon metric by Theorem~\ref{sG-geometry},
as it must happen according to Popovici's result
\cite[Proposition 4.1]{Pop09}.
We recall that there exists \cite{COUV} a holomorphic family $(N,J_a)$, $N$ being a 6-dimensional nilmanifold,
such that $(N,J_a)$ has balanced metrics for any $a\not=0$, but the central limit $(N,J_0)$ does not admit
strongly Gauduchon metric.

\section{The classification}\label{sec-preli}

\noindent We first show that the existence of a holomorphic form of bidegree $(n,0)$
with respect to an invariant complex structure on a $2n$-dimensional solvmanifold implies the
existence of an invariant non-zero closed $(n,0)$-form. Furthermore:

\begin{proposition}\label{holglobalform}
Let $M=\Gamma\backslash G$ be a $2n$-dimensional solvmanifold
endowed with an invariant complex structure $J$.
If $\Omega$ is a nowhere vanishing holomorphic $(n,0)$-form
on $(M,J)$, then $\Omega$ is necessarily invariant.
\end{proposition}

\begin{proof}
Since $J$ is invariant, we consider a global basis $\{\omega^1,\ldots,\omega^n\}$ of
invariant $(1,0)$-forms on $(M,J)$. Then, there is a nowhere vanishing complex-valued function
$f\colon M\rightarrow \mathbb{C}$ such that $\Omega=f\,\omega^1\wedge\cdots\wedge\omega^n$.
Since $\Omega$ is holomorphic, we have
$\db \Omega = \db f \wedge \omega^1\wedge\cdots\wedge\omega^n + f\, \db (\omega^1\wedge\cdots\wedge\omega^n) =0$, that is,
$\db (\omega^1\wedge\cdots\wedge\omega^n) = - \db (\log f)\wedge \omega^1\wedge\cdots\wedge\omega^n$.
The latter form is an invariant $(n,1)$-form on $(M,J)$, so there is an invariant form $\alpha$
of bidegree $(0,1)$ on $(M,J)$
such that
\begin{equation}\label{1}
\db (\log f) =\alpha.
\end{equation}
Now we apply the well-known symmetrization process, which assigns an invariant $k$-form $\widetilde\gamma$ to any
$k$-form $\gamma$ on $M$ by integration with respect to a volume
element on $M$ induced by a bi-invariant one on the Lie group
$G$ (its existence is guaranteed by \cite{Milnor}).
By extending the symmetrization to the space of complex forms, since $J$ is invariant we have that
symmetrization preserves the
bidegree and commutes with $\db$, because it commutes with the differential $d$ (see \cite{COUV,FG} for more details).
Applying this to \eqref{1}, we get
$$
\db (\widetilde{\log f}) =\widetilde\alpha=\alpha,
$$
because $\alpha$ is invariant.
But $\widetilde{\log f}$ is the symmetrization of a function,
so it is a constant and then $\db (\widetilde{\log f}) = 0$.
Therefore, $\alpha=0$ and by \eqref{1} we get $\db (\log f)=0$.
This means that $\log f$ is a holomorphic function
on a compact complex manifold, which implies that
$\log f=c$, where $c$ is a constant.
In conclusion, $f=\exp(c)$ is a constant function, and $\Omega$ is necessarily invariant.
\end{proof}

As a consequence of this result, in order to
describe the solvmanifolds $M=\Gamma\backslash G$ admitting an invariant complex structure
with holomorphically trivial canonical bundle, we are led to study the unimodular solvable Lie algebras
$\frak g$ admitting a complex structure $J$ with non-zero closed $(n,0)$-form.
Next we classify such Lie algebras in dimension 6.

We will recall first the formalism of stable forms which enable us to construct the space
of almost complex structures on a given Lie algebra \cite{C,H}.
Let $(V,\nu)$ be an oriented six-dimensional vector space, a $3$-form $\rho$
is \emph{stable} if its orbit under the action of the group GL($V$) is open.
Let $\kappa\colon\Lambda^5V^*\longrightarrow V$ be the isomorphism defined by
$\kappa(\eta)=X$ where $X$ is such that $\iota_X\nu=\eta$, and let $K_\rho:V\longrightarrow V$
be the endomorphism given by $K_\rho(X)=\kappa(\iota_X\rho\wedge\rho)$.
The square of this endomorphism is proportional to the identity map so
it is an automorphism of the vector space.
Let $\lambda(\rho)$ be the constant of proportionality, that is,
$$
\lambda(\rho)=\frac{1}{6}\trace(K_\rho^2).
$$
Then, the stability of $\rho$ is equivalent to the open condition $\lambda(\rho)\neq 0$.
When $\lambda(\rho)<0$ the endomorphism $J_\rho:=\frac{1}{\sqrt{|\lambda(\rho)|}}K_\rho$
gives rise to an almost complex structure on $V$.
Moreover, $\lambda(\rho)$ enables to construct a specific volume form
$\phi(\rho):=\sqrt{|\lambda(\rho)|}\nu\in\Lambda^6V^*$ such that
the action of the dual endomorphism $J_\rho^*$ on 1-forms is given by the formula
\begin{equation}\label{def-J}
\left( (J_\rho^*\alpha)(X) \right) \phi(\rho)=\alpha\wedge\iota_X\rho\wedge\rho,
\end{equation}
for any $\alpha\in\frg^*$ and $X\in\frg$.

As a consequence there is a natural mapping
\begin{equation}\label{correspondencia}
\{\rho\in\Lambda^3V^*\ |\  \lambda(\rho)<0\}\rightarrow\{J\colon V\rightarrow V\ |\ J^2=-I\}
\end{equation}
assigning each $\rho$ to $J=J_{\rho}$ through relation \eqref{def-J}.
Although this map is not injective, as for example $J_{\rho'}=J_\rho$ when $\rho'$ is proportional to $\rho$,
it is onto and therefore it covers the space of almost complex structures on $V$.

We are interested in the complex structures on a Lie algebra $\frg$ admitting non-zero closed (3,0)-form.
Let $Z^3(\frg)=\{\rho\in\Lambda^3\frg^* \mid d\rho=0 \}$.
The map \eqref{correspondencia} restricts to the surjective mapping
$$
\{\rho\in Z^3(\frg)\ |\ \lambda(\rho)<0,\ d(J_\rho^*\rho)=0\}\rightarrow\{J\colon\frg\rightarrow\frg\ |\ J^2=-I,\
\exists\Psi\in\Lambda^{3,0}\ \text{closed}\}.
$$
The closed (3,0)-form is given by $\Psi=\rho+iJ_\rho^*\rho$.

The next result provides an equivalent condition to determine the existence of such complex structures on $\frg$.

\begin{lemma}\label{prop}
Let $\frg$ be a Lie algebra and $\nu$ a volume form on $\frg$. Then, $\frg$ admits an almost complex structure
with a non-zero closed $(3,0)$-form if and only if there exists $\rho\in Z^3(\frg)$ such that the endomorphism
$\tilde J_\rho^*\colon\frg^*\rightarrow\frg^*$ defined by
\begin{equation}\label{criterio}
\left( (\tilde J_\rho^*\alpha)(X) \right) \nu=\alpha\wedge\iota_X\rho\wedge\rho,
\end{equation}
for any $\alpha\in\frg^*$ and $X\in\frg$, satisfies that $\tilde J_\rho^*\rho$ is closed and ${\rm tr}(\tilde J_\rho^{*2})<0$.
\end{lemma}

\begin{proof}
Let $J\colon \frg\longrightarrow\frg$ be an almost complex structure admitting
a non-zero (3,0)-form $\Psi=\Psi_++i\Psi_-$ which is closed.
Let $\rho=\Psi_+$. Then, $\lambda(\rho)<0$, $J=J_{\rho}$ is determined
by~\eqref{def-J} and the form $J_\rho^*\rho=\Psi_-$ is closed.
Since the associated form $\phi(\rho)$ is a volume form on $\frg$,
we have that $\nu=s\,\phi(\rho)$ for some $s\not=0$.
Now, for the endomorphism $\tilde J_\rho^*\colon \frg^*\longrightarrow\frg^*$
given by~\eqref{criterio}
we get
$$
s \left( (\tilde J_\rho^*\alpha)(X) \right) \phi(\rho)=
\left( (\tilde J_\rho^*\alpha)(X) \right) \nu =
\alpha\wedge\iota_X\rho\wedge\rho =
\left( (J_\rho^*\alpha)(X)  \right) \phi(\rho),
$$
for any $\alpha\in\frg^*$ and $X\in\frg$.
This implies that $J_\rho^*=s\tilde J_\rho^*$. Therefore, $\text{tr}(\tilde J_\rho^{*2})<0$
if and only if $\text{tr}(J_\rho^{*2})<0$,
and moreover, $d(\tilde J_\rho^*\rho)=0$ if and only if $d(J_\rho^*\rho)=0$.
\end{proof}

For the computation of the endomorphism $\tilde J_\rho^*$ we will use the simplest
volume form $\nu=e^{123456}$,
where $\{e^1,\ldots,e^6\}$ is the basis of $\frg^*$ in which the Lie algebra is expressed.
Notice that we will follow the notation given in \cite{F-Schu1,F-Schu2,Schu} and \cite{T}
to name the Lie algebras; for instance, the notation $\fre(2)\oplus\fre(1,1) = (0,-e^{13},e^{12},0,-e^{46},-e^{45})$
means that $\fre(2)\oplus\fre(1,1)$ is the (decomposable) Lie algebra determined by a basis $\{e^i\}_{i=1}^6$ such that
$de^1=0$, $de^2=-e^{13}$, $de^3=e^{12}$, $de^4=0$, $de^5=-e^{46}$, $de^6=-e^{45}$.

The next two concrete examples show how we will proceed in general in the proofs
of Propositions~\ref{decomposable} and~\ref{indecomposable} below
in order to exclude candidates.
Some of the computations in these examples, as well as in Propositions~\ref{decomposable} and~\ref{indecomposable}
and in Section~\ref{sec-reduc}, were carried out using Mathematica and the differential forms package scalarEDC
by S. Bonanos available at www.inp.demokritos.gr/˜sbonano/.

\begin{example}\label{ex1}
{\rm
Let us consider the indecomposable solvable Lie algebra $\frg=A_{6,25}^{0,-1}=(e^{23},e^{26},-e^{36},0,e^{46},0)$.
Any $\rho\in Z^3(\frg)$ is given by
$$
\begin{array}{l}
\rho=a_1e^{123}+a_2e^{126}+a_3e^{136}+a_4e^{234}+a_5(e^{235}-e^{146})+a_6e^{236}+a_7e^{246}\\[4pt]
\quad\, +a_8e^{256}+a_9e^{346}+a_{10}e^{356}+a_{11}e^{456},
\end{array}
$$
for $a_1,\ldots,a_{11}\in \mathbb{R}$. Fix $\nu=e^{123456}$, and let $\tilde J_\rho^*$
be the endomorphism given by~\eqref{criterio}. A direct calculation shows that
$$
\text{tr}(\tilde J_\rho^{*2})=6(a_5^2-a_1 a_{11})^2 \geq 0.
$$
In this case it is not worth evaluating the closedness of $\tilde J_\rho^*\rho$ because by Lemma~\ref{prop}
there is no almost complex structure $J_\rho^*$ coming from a closed 3-form $\rho\in Z^3(\frg)$
and in particular $\frg$ does not admit a closed complex volume form.
}
\end{example}

\begin{example}\label{ex2}
{\rm
Let us consider the $5\oplus1$ decomposable solvable Lie algebra
$\frg=A_{5,15}^{-1}\oplus\R=(e^{15}+e^{25},e^{25},-e^{35}+e^{45},-e^{45},0,0)$.
Any $\rho\in Z^3(\frg)$ is given by
$$
\begin{array}{l}
\rho=a_1e^{125}+a_2e^{135}+a_3e^{145}+a_4e^{156}+a_5e^{235}+a_6(e^{236}-e^{146})+a_7e^{245}\\[4pt]
\quad\, +a_8e^{246}+a_9e^{256}+a_{10}e^{345}+a_{11}e^{356}+a_{12}e^{456},
\end{array}
$$
for $a_1,\ldots,a_{12}\in \mathbb{R}$. Take $\nu=e^{123456}$, and let $\tilde J_\rho^*$
be the endomorphism given by~\eqref{criterio}.
Then, we have
$$
\frac{1}{6}\text{tr}(\tilde J_\rho^{*2})=
(a_3+a_5)^2a_6^2+4(a_1 a_{10}-a_2 a_{7})a_6^2 -2(a_3-a_5)a_2 a_6 a_8 +a_2^2 a_8^2
$$
and
$$
\begin{array}{l}
d(\tilde J_\rho^*\rho)=
2 a_6^2 \left( 2a_1 e^{1256}+ a_2 (e^{1456}+e^{2356})+ (a_3+a_5) e^{2456}-2 a_{10}e^{3456} \right).
\end{array}
$$
Since the form $\tilde J_\rho^*\rho$ must be closed, we distinguish two cases depending on the vanishing of the coefficient $a_6$.
If $a_6=0$ then $\text{tr}(\tilde J_\rho^{*2})=6(a_2a_8)^2\geq 0$, and if $a_6\neq 0$ then $a_1=a_2=a_3+a_5=a_{10}=0$ and so
$\text{tr}(\tilde J_\rho^{*2})=0$. Consequently, Lemma~\ref{prop} assures that there is no almost complex structure $J_\rho^*$
admitting a non-zero closed (3,0)-form.
}
\end{example}

From now on we shall denote by $b_k(\frg)$ the dimension of the $k$-th Chevalley-Eilenberg cohomology
group $H^k(\frg)$ of the Lie algebra.
The next lemma shows a simple but useful obstruction to the existence of complex structures
with non-zero closed $(3,0)$-volume forms in the unimodular case involving $b_3$.
We remind that the unimodularity of $\frg$ is equivalent to $b_6(\frg)=1$.

\begin{lemma}\label{lema}
If  $\frg$ is an unimodular Lie algebra admitting a complex structure with a non-zero closed $(3,0)$-form $\Psi$,
then $b_3(\frg)\geq 2$.
\end{lemma}

\begin{proof}
Let $\Psi_+,\Psi_-\in\Lambda^3(\frg^*)$ be the real and imaginary parts of $\Psi$,
that is, $\Psi=\Psi_++i\,\Psi_-$. Since $\Psi$ is closed
we have that $d(\Psi_+)=d(\Psi_-)=0$ and therefore $[\Psi_+],[\Psi_-]\in H^3(\frg)$.
It is sufficient to see that both classes are non-zero
and, moreover, that they are not cohomologous.

Suppose that there exist $a,b\in\mathbb{R}$ with $a^2+b^2\not=0$ such that
$a\Psi_+ + b\Psi_-=d\alpha$ for some $\alpha\in\Lambda^2(\frg^*)$.
Since $0\neq\frac{i}{2}\Psi\wedge\bar\Psi=\Psi_+\wedge\Psi_-\in\Lambda^6(\frg^*)$,
we get
$$
d(\alpha\wedge(-b\Psi_+ + a\Psi_-))=
(a\Psi_+ + b\Psi_-)\wedge(-b\Psi_+ + a\Psi_-)=
(a^2+b^2)\Psi_+\wedge\Psi_-\not=0.
$$
But this is in contradiction to the unimodularity of $\frg$.
\end{proof}

As a consequence of Lemma~\ref{lema} we will concentrate on unimodular
(non nilpotent) solvable Lie algebras $\frak g$ with $b_3 ({\frak g})\geq 2$.
We will tackle the classification problem first in the decomposable case.

Let $\frg=\frb\oplus\frc$. 
The unimodularity and solvability of $\frg$ and Lemma~\ref{lema} imply restrictions on the factors.
In fact, $\frg$ is unimodular, resp. solvable, if and only if $\frb$
and $\frc$ are unimodular, resp. solvable.
Moreover, by Lemma~\ref{lema} and the well known formula relating the cohomology
of $\frg$ with the cohomologies of the factors,
we have
\begin{equation}\label{b3}
b_3(\frb)b_0(\frc)+b_2(\frb)b_1(\frc)+b_1(\frb)b_2(\frc)+b_0(\frb)b_3(\frc)=b_3(\frg)\geq 2.
\end{equation}

\begin{proposition}\label{decomposable}
Let $\frg=\frb\oplus\frc$ be a six-dimensional decomposable unimodular (non nilpotent)
solvable Lie algebra admitting a complex structure with a non-zero closed $(3,0)$-form.
Then, $\frg$ is isomorphic to $\fre(2)\oplus\fre(1,1)$, $A_{5,7}^{-1,-1,1}\oplus\R$
or $A_{5,17}^{\alpha,-\alpha,1}\oplus\R$ with $\alpha\geq 0$.
\end{proposition}

\begin{proof}
Since $\frg$ is decomposable, we divide the proof in the three cases $3\oplus3$, $4\oplus2$ and $5\oplus1$.
In the $3\oplus3$ case the inequality \eqref{b3} is always satisfied.
The $3\oplus3$ decomposable unimodular (non nilpotent) solvable Lie algebras are $\fre(2)\oplus\fre(2)$, $\fre(2)\oplus\fre(1,1)$,  $\fre(2)\oplus\frh_3$, $\fre(2)\oplus\R^3$, $\fre(1,1)\oplus\fre(1,1)$, $\fre(1,1)\oplus\frh_3$ and $\fre(1,1)\oplus\R^3$
(see Table 1 in the Appendix for a description of the Lie algebras).
An explicit computation shows that there is no $\rho\in Z^3$ satisfying the conditions $\lambda(\rho)<0$ and $d(J_\rho^*\rho)=0$,
except for $\frg=\fre(2)\oplus\fre(1,1)$.
We give an example of a closed complex volume form
for $\fre(2)\oplus\fre(1,1)$ in Table 1.

Since $\R^2$ is the only 2-dimensional unimodular Lie algebra, the $4\oplus2$ case is reduced to the study of $\frg=\frb\oplus\R^2$
for any 4-dimensional unimodular (non nilpotent) solvable Lie algebra $\frb$ satisfying \eqref{b3}, i.e.
$b_3(\frb)+2b_2(\frb)+b_1(\frb)\geq 2$. The resulting Lie algebras are: $A_{4,2}^{-2}\oplus\R^2$, $A_{4,5}^{\alpha,-1-\alpha}\oplus\R^2$ with $-1<\alpha\leq-\frac{1}{2}$, $A_{4,6}^{\alpha,-\frac{\alpha}{2}}\oplus\R^2$, $A_{4,8}\oplus\R^2$ and $A_{4,10}\oplus\R^2$ (see Table 1).
However, all of them satisfy $\lambda(\rho)\geq 0$ for any $\rho\in Z^3(\frg)$.

Finally, the $5\oplus 1$ case consists of Lie algebras of the form $\frg=\frb\oplus\R$ for any 5-dimensional unimodular (non nilpotent) solvable
Lie algebra $\frb$ such that $(b_2(\frb),b_3(\frb))\neq (0,0),(1,0),(0,1)$. Therefore, the Lie algebras are: $A_{5,7}^{-1,-1,1}\oplus\R$, $A_{5,7}^{-1,\beta,-\beta}\oplus\R$ with $0<\beta<1$, $A_{5,8}^{-1}\oplus\R$, $A_{5,9}^{-1,-1}\oplus\R$, $A_{5,13}^{-1,0,\gamma}\oplus\R$ with $\gamma>0$, $A_{5,14}^0\oplus\R$, $A_{5,15}^{-1}\oplus\R$, $A_{5,17}^{0,0,\gamma}\oplus\R$ with $0<\gamma<1$, $A_{5,17}^{\alpha,-\alpha,1}$ with $\alpha\geq 0$, $A_{5,18}^0\oplus\R$, $A_{5,19}^{-1,2}\oplus\R$, $A_{5,19}^{1,-2}\oplus\R$, $A_{5,20}^0\oplus\R$, $A_{5,26}^{0,\pm1}\oplus\R$, $A_{5,33}^{-1,-1}\oplus\R$ and $A_{5,35}^{0,-2}\oplus\R$.
The explicit computation of each case allows us to distinguish the following three situations:\\[-8pt]

\noindent $\bullet$ If $\frg=A_{5,9}^{-1,-1}\oplus\R$ or $A_{5,26}^{0,\pm 1}\oplus\R$, then $\lambda(\rho)\geq 0$ for all $\rho\in Z^3(\frg)$.\\[-8pt]

\noindent $\bullet$ The Lie algebras $A_{5,7}^{-1,-1,1}\oplus\R$ and $A_{5,17}^{\alpha,-\alpha,1}\oplus\R$
with $\alpha\geq 0$ admit a closed complex volume form (see Table 1 for a concrete example).\\[-8pt]

\noindent $\bullet$ For the rest of Lie algebras there is no $\rho\in Z^3(\frg)$ satisfying $d(J_\rho^*\rho)=0$ and $\lambda(\rho)<0$
simultaneously.\\[-8pt]

\noindent In conclusion, in the $5\oplus 1$ case the only possibilities are $A_{5,7}^{-1,-1,1}\oplus\R$
and the family $A_{5,17}^{\alpha,-\alpha,1}\oplus\R$ with $\alpha\geq 0$.
\end{proof}


Next we obtain the classification in the indecomposable case.

\begin{proposition}\label{indecomposable}
Let $\frg$ be a six-dimensional indecomposable unimodular (non nilpotent) solvable Lie algebra
admitting a complex structure with a non-zero closed $(3,0)$-form. Then, $\frg$ is isomorphic to $N_{6,18}^{0,-1,-1}$,
$A_{6,37}^{0,0,1}$, $A_{6,82}^{0,1,1}$, $A_{6,88}^{0,0,1}$, $B_{6,4}^1$ or $B_{6,6}^1$.
\end{proposition}

\begin{proof}
The Lie algebras $\frg$ such that $b_3(\frg)\geq 2$ are listed in Table 2 in the Appendix.
The indecomposable case is long to analyze because of the amount of Lie algebras,
but after performing the computations we distinguish the following three situations:\\[-8pt]

\noindent $\bullet$ Let $\frg$ be one of the following Lie algebras: $A_{6,13}^{a,-2a,2a-1}$ ($a\in \mathbb{R}-\{-1,0,\frac{1}{3},\frac{1}{2}\}$), $A_{6,13}^{a,-a,-1}$ ($a>0$, $a\neq 1$), $A_{6,14}^{\frac{1}{3},-\frac{2}{3}}$, $A_{6,18}^{a,b}$ with $(a,b)\in\{(-\frac{1}{2},-2),(-2,1)\}$, $A_{6,25}^{a,b}$ with $(a,b)\in\{(0,-1),(-\frac{1}{2},-\frac{1}{2})\}$, $A_{6,32}^{0,b,-b}$ ($b>0$), $A_{6,34}^{0,0,\epsilon}$ ($\epsilon=0,1)$, $A_{6,35}^{a,b,c}$ with $a>0$ and $(b,c)\in\{(-2a,a),(-a,0)\}$ and $A_{6,37}^{0,0,c}$ ($c>0$, $c\neq1$).
Then, $\lambda(\rho)\geq 0$ for any $\rho\in Z^3(\frg)$.\\[-8pt]

\noindent $\bullet$ The Lie algebras $N_{6,18}^{0,-1,-1}, A_{6,37}^{0,0,1}, A_{6,82}^{0,1,1}, A_{6,88}^{0,0,1}, B_{6,4}^1$ and $B_{6,6}^{1}$ admit a non-zero closed (3,0)-form (see Table 2 for a concrete example).\\[-8pt]

\noindent $\bullet$ For the rest of Lie algebras there is no $\rho\in Z^3(\frg)$ such that $d(J_\rho^*\rho)=0$ and $\lambda(\rho)<0$.
\end{proof}

From Propositions~\ref{decomposable} and~\ref{indecomposable} it follows the following
classification:

\begin{theorem}\label{main-th}
Let $\frg$ be an unimodular (non nilpotent) solvable Lie algebra of dimension $6$.
Then, $\frg$ admits a complex structure with a non-zero closed $(3,0)$-form
if and only if it is isomorphic to one in the following list:
$$
\begin{array}{lll}\label{teorema}
&&\frg_1=A_{5,7}^{-1,-1,1}\!\oplus\R = (e^{15},-e^{25},-e^{35},e^{45},0,0) ,\\[4pt]
&&\frg_2^{\alpha}=A_{5,17}^{\alpha,-\alpha,1}\!\oplus\R= (\alpha e^{15}\!\!+\!e^{25},-e^{15}\!\!+\!\alpha e^{25},-\alpha e^{35}\!\!+\!e^{45},
-e^{35}\!\!-\!\alpha e^{45},0,0),
\ \alpha\geq 0,\\[4pt]
&&\frg_3=\fre(2)\oplus\fre(1,1) = (0,-e^{13},e^{12},0,-e^{46},-e^{45}) ,\\[4pt]
&&\frg_4=A_{6,37}^{0,0,1} = (e^{23},-e^{36},e^{26},-e^{56},e^{46},0) ,\\[4pt]
&&\frg_5=A_{6,82}^{0,1,1} = (e^{24}+e^{35},e^{26},e^{36},-e^{46},-e^{56},0) ,\\[4pt]
&&\frg_6=A_{6,88}^{0,0,1} = (e^{24}+e^{35},-e^{36},e^{26},-e^{56},e^{46},0) ,\\[4pt]
&&\frg_7=B_{6,6}^{1} = (e^{24}+e^{35},e^{46},e^{56},-e^{26},-e^{36},0) ,\\[4pt]
&&\frg_8=N_{6,18}^{0,-1,-1} = (e^{16}-e^{25},e^{15}+e^{26},-e^{36}+e^{45},-e^{35}-e^{46},0,0) ,\\[4pt]
&&\frg_9=B_{6,4}^1 = (e^{45},e^{15}+e^{36},e^{14}-e^{26}+e^{56},-e^{56},e^{46},0) .
\end{array}
$$
\end{theorem}

From now on, we will use only the simplified notation $\frg_1,\frg_2^{\alpha},\frg_3,\ldots,\frg_9$
when refering to the Lie algebras listed in the previous theorem.

\subsection{Existence of lattices}\label{subsect-lattices}

In this section we show that the simply-connected solvable Lie groups $G_k$ corresponding to the Lie algebras $\frg_k$
in Theorem~\ref{main-th} admit lattices $\Gamma_k$ of maximal rank. Therefore, we get compact complex
solvmanifolds $\Gamma_k\backslash G_k$ with holomorphically trivial canonical bundle.

Let $H$ be a $n$-dimensional nilpotent Lie group. We remind that a  connected and simply-connected Lie group $G$ with nilradical $H$ is called  \emph{almost nilpotent} (resp. \emph{almost abelian}) if it can be written as $G=\R\ltimes_\mu H$ (resp. $G=\R\ltimes_\mu \R^n$, that is, $H=\R^n$).
If we denote by $e$ the identity element of $H$ then following \cite{Bock} we have that $d_e(\mu(t))=\exp^{{\rm GL}(n,\R)}(t\varphi)$, $\varphi$ being a certain derivation of the Lie algebra $\frh$ of $H$. Although it is not easy in general to know whether a solvable Lie group $G$ admits a lattice, the next result allows to construct one in the case that $G$ is almost nilpotent.

\begin{lemma}\label{lema_lattice}\cite{Bock}
Let $G=\R\ltimes_\mu H$ be a $(n+1)$-dimensional almost nilpotent Lie group with nilradical $H$ and $\frh$ the Lie algebra of $H$.
If there exists $0\neq t_1\in\R$ and a rational basis $\{X_1,\ldots,X_n\}$ of $\mathfrak{h}$ such that the coordinate matrix of $d_e(\mu(t_1))$ in such basis is integer, then $\Gamma=t_1\mathbb{Z}\ltimes_\mu \exp^H(\mathbb{Z}\langle X_1,\ldots,X_n\rangle)$ is a lattice in $G$.
\end{lemma}

Now we will use the former lemma to prove the following result:

\begin{proposition}\label{existence-of-lattices}
For any $k\not=2$, the connected and simply-connected Lie group $G_k$ with underlying Lie algebra $\frg_k$ admits a lattice.

For $k=2$, there exists a countable number of distinct $\alpha$'s, including $\alpha=0$, for which the
connected and simply-connected Lie group with underlying Lie algebra $\frg_2^{\alpha}$ admits a lattice.
\end{proposition}

\begin{proof}
The Lie algebra $\frg_8$ is not almost nilpotent, but its corresponding connected and simply-connected Lie group $G_8$
admits a lattice by \cite{Y}. It is not hard to see that for $k\neq8$ the Lie algebra $\frg_k$ of Theorem~\ref{main-th}
is either almost nilpotent or a product of almost nilpotent Lie algebras.
In fact, we find the following correspondence with some of the Lie algebras studied in \cite{Bock}
(we use the notation in that paper in order to compare directly with the Lie algebras therein): $\frg_1\cong \frg_{5,7}^{-1,-1,1}\oplus\R$, $\frg_2^{0}\cong \frg_{5,17}^{0,0,1}\oplus\R$, $\frg_3\cong\frg_{3,5}^0\oplus\frg_{3,4}^{-1}$, $\frg_4\cong\frg_{6,37}^{0,0,-1}$, $\frg_5\cong\frg_{6,88}^{0,-1,0}$, $\frg_6\cong\frg_{6,92}^{0,-1,-1}$ and $\frg_7\cong\frg_{6,92}^*$.
The simply connected Lie group  $G_3$ admits a lattice, since it is product of two  $3$-dimensional  Lie groups and  every     $3$-dimensional  completely solvable simply  connected Lie group has a lattice.
For $\frak g_1,  \frg_2^{0}, \frak g_4, \frak g_5, \frak g_6$ and $\frak g_7$, the existence of lattices in the corresponding Lie groups is already proved in~\cite{Bock}. In fact, a lattice is respectively given by:
$$
\begin{array}{l}
\Gamma_1 = (t_1 \Z \ltimes_{\mu} exp^{H} ({\Z} \langle e_1, \ldots, e_4\rangle) )\times \Z,  \, \,  H = \R^4, \, t_1 =  \log \left ( \frac{3 + \sqrt{5}}{2} \right );\\[3pt]
\Gamma_2^0 =  (t_1 \Z \ltimes_{\mu} exp^{H} ({\Z} \langle e_1, \ldots, e_4\rangle) )\times \Z,  \, \,   H = \R^4, \,  t_1 =  \pi;\\[3pt]
\Gamma_4 =   t_1 \Z \ltimes_{\mu} exp^{H} ({\Z} \langle e_1, \ldots, e_5\rangle),  \, \,  H =  {\mbox {Heis}}_3 \times \R^2,  \,  t_1 = \pi;\\[3pt]
\Gamma_5 =  t_1 \Z \ltimes_{\mu} exp^{H} ({\Z} \langle  - \frac{1}{\sqrt 5} e_1,  \frac{1}{\sqrt 5} (e_3 - e_5),  \alpha e_3 - \frac{1}{5 \alpha}  e_5,   \frac{1}{\sqrt 5} (e_2 - e_4),  \alpha e_2 - \frac{1}{5 \alpha} e_4\rangle), \\[1pt]
 \alpha =  \frac{\sqrt{5}(3 + \sqrt{5})}{10}, \, \, H = {\mbox {Heis}}_5, \,  t_1 = \log\left (\frac{3+\sqrt 5}{2} \right);\\[3pt]
\Gamma_6 =    t_1 \Z \ltimes_{\mu} exp^{H} ({\Z} \langle e_1, \ldots, e_5\rangle),  \, \,  H = {\mbox {Heis}}_5,  \, t_1 = \pi; \\[3pt]
\Gamma_7 =   t_1 \Z \ltimes_{\mu} exp^{H} ({\Z} \langle 2 e_1 , -e_2 - e_3 , e_4 + e_5 , e_2- e_3 , e_4 - e_5 \rangle),  \, \,  H =  {\mbox {Heis}}_5,  \, t_1 = \pi, \\[3pt]
\end{array}
$$
 where  ${\mbox {Heis}}_3$ and ${\mbox {Heis}}_5$ are   the real    Heisenberg group of dimension $3$ and $5$.
So, it remains to study $\frg_2^\alpha$ with $\alpha>0$, and $\frg_9$.

We will show first that there exists a countable subfamily of $\frg_2^\alpha$ with $\alpha>0$ whose
corresponding Lie group $G_2^\alpha$ admits lattice. The $5$-dimensional factor $A_{5,17}^{\alpha,-\alpha,1}$ in the decomposable Lie algebra $\frg_2^\alpha=A_{5,17}^{\alpha,-\alpha,1}\oplus\R$ is given by
$$
[e_1,e_5]=-\alpha e_1+e_2,\;[e_2,e_5]=-e_1-\alpha e_2,\; [e_3,e_5]=\alpha e_3+e_4,\;[e_4,e_5]=-e_3+\alpha e_4,
$$
which is almost abelian since $A_{5,17}^{\alpha,-\alpha,1}=\R\ltimes_{\text{ad}_{e_5}}\R^4$. If we denote by  $B_\alpha$ the coordinate matrix of the derivation $\text{ad}_{e_5}\colon\R^4\to\R^4$ in the basis $\{e_1,e_2,e_3,e_4 \}$, then the coordinate matrix of $d_e(\mu(t))$ is the exponential 
\begin{equation*}
e^{tB_\alpha}=\left(\begin{array}{cccc}e^{\alpha t}\text{cos}(t)&e^{\alpha t}\text{sin}(t)&0&0\\-e^{\alpha t}\text{sin}(t)&e^{\alpha t}\text{cos}(t)&0&0\\0&0&e^{-\alpha t}\text{cos}(t)&e^{-\alpha t}\text{sin}(t)\\0&0&-e^{-\alpha t}\text{sin}(t)&e^{-\alpha t}\text{cos}(t)\end{array}\right).
\end{equation*}
If $t_l=l\pi$ with $l\in\mathbb{Z}$ and $l>0$, then the characteristic polynomial of the matrix $e^{t_lB_\alpha}$ is $p(\lambda)=(1-(-1)^l(e^{\alpha t_l}+e^{-\alpha t_l})\lambda+\lambda^2)^2$, which is integer if $\alpha_{l,m}=\frac{1}{l\pi}\text{log}(\frac{m+\sqrt{m^2-4}}{2})$ with $m\in\Z$ and $m>2$. Moreover, $e^{t_lB_\alpha}=P^{-1}C_{l,m}P$, where
$$
P^{-1}=\left(\begin{array}{cccc}0&0&\epsilon&\beta^+\\\epsilon&\beta^+&0&0\\0&0&-\epsilon&\beta^-\\-\epsilon&\beta^-&0&0\end{array}\right),\quad\quad C_{l,m}=\left(\begin{array}{cccc}0&-1&0&0\\1&m(-1)^l&0&0\\0&0&0&-1\\0&0&1&m(-1)^l\end{array}\right),
$$
with $\epsilon=\frac{1}{\sqrt{m^2-4}}$ and $\beta^\pm=\frac{m^2-4\pm(-1)^l\sqrt{m^2-4}}{2(m^2-4)}$. Taking the basis $X_1=\epsilon(e_2-e_4)$, $X_2=\beta^+e_2+\beta^-e_4$, $X_3=\epsilon(e_1+e_3)$ and $X_4=\beta^+e_1+\beta^-e_3$ of $\R^4$ and using Lemma ~\ref{lema_lattice} we have that $\Gamma'=l\pi\Z\ltimes_\mu \Z\langle X_1,\ldots,X_4\rangle$ is a lattice of the simply-connected Lie group associated to      $A_{5,17}^{\alpha,-\alpha,1}$ with $\alpha=\alpha_{l,m}$.
Hence, $\Gamma=\Gamma'\times\Z$ is a lattice in $G_2^{\alpha_{l,m}}$.

The Lie algebra $\frg_9$ can be seen as an almost nilpotent Lie algebra $\frg=\R\ltimes_{\text{ad}_{e_6}}\frh$,
where $\frh=\langle e_1,\ldots,e_5\, |\, [e_1,e_4]=-e_3,\, [e_1,e_5]=-e_2,\, [e_4,e_5]=-e_1\rangle$
is a $5$-dimensional nilpotent Lie algebra.
Proceeding in a similar manner as for $\frg_2^\alpha$, we compute the characteristic polynomial of $d_e(\mu(t))$ getting that $p(\lambda)=(\lambda^2-2\lambda\text{cos}(t)+1)^2$.
If $t_1=\pi$ then $p(\lambda)\in\mathbb{Z}[\lambda]$ and the coordinate matrix of $d_e(\mu(t_1))$ in the basis $X_1=\frac{\pi}{2}e_1$, $X_2=\sqrt{\frac{\pi}{2}}e_4$, $X_3=\sqrt{\frac{\pi}{2}}e_5$, $X_4=(\frac{\pi}{2})^{3/2}e_2+\sqrt{\frac{\pi}{2}}e_4$ and $X_5=-(\frac{\pi}{2})^{3/2}e_3+\sqrt{\frac{\pi}{2}}e_5$ of $\frh$ is
$$
C=\left(\begin{array}{ccccc}1&0&0&0&0\\0&-2&0&-1&0\\0&0&-2&0&-1\\0&1&0&0&0\\0&0&1&0&0\end{array}\right).
$$
Moreover, $\{X_1,\ldots,X_5\}$ is a rational basis of $\frh$ because $[X_1,X_2]=[X_1,X_4]=-X_3+X_5$, $[X_1,X_3]=[X_1,X_5]=X_2-X_4$, $[X_2,X_3]=[X_2,X_5]=-[X_3,X_4]=-X_1$.
Hence, if we denote by $H$ the simply-connected Lie group corresponding to $\frh$, then using Lemma~\ref{lema_lattice} we have that $\Gamma=\pi\mathbb{Z}\ltimes_\mu\text{exp}^H(\mathbb{Z}\langle X_1,\ldots,X_5\rangle)$ is a lattice
in the Lie group~$G_9$.
\end{proof}

\begin{remark}
{\rm
Bock found a lattice for the Lie group associated to $A_2^{\alpha,-\alpha,1}$ with $\alpha=\alpha_{1,3}=\frac{1}{\pi}\log\frac{3+\sqrt{5}}{2}$,
that is, for $l=1$ and $m=3$. Notice that our result for $k=2$ is consistent with the result obtained by Witte in \cite[Prop. 8.7]{Wi},
where it is showed that only countably many non-isomorphic simply-connected Lie groups admit a lattice,
so that one cannot expect a lattice to exist for any real $\alpha>0$.

The Lie algebra $\frg_9$ does not appear in \cite{Bock} because its nilradical is the $5$-dimensional Lie algebra $\frh$, which is isomorphic to $\frg_{5,3}$ (in the notation of \cite{Bock}), but there are only two solvable and unimodular Lie algebras with nilradical $\frg_{5,3}$
considered in that paper (namely $\frg_{6,76}^{-1}$ and $\frg_{6,78}$) which are both completely solvable, but $\frg_9$ is not.
}
\end{remark}

\section{Moduli of complex structures}\label{sec-reduc}

\noindent In this section we classify, up to equivalence, the complex structures having closed (3,0)-form on the Lie algebras of
Theorem~\ref{main-th}.
Recall that two complex structures $J$
and $J'$ on $\frg$ are said to be \emph{equivalent} if there exists an automorphism
$F\colon \frg\longrightarrow\frg$ such that $F\circ J=J'\circ F$.

\subsection{The decomposable case}

We consider firstly the $5\oplus1$ decomposable Lie algebras.

\begin{lemma}\label{1-forma-cerrada}
Let $J$ be any complex structure on $\frg_1$ or $\frg_2^{\alpha}$, $\alpha\geq 0$,
with closed volume $(3,0)$-form, then there is a non-zero closed $(1,0)$-form.
More concretely, the $(1,0)$-form $e^5-iJ^*e^5$ is closed.
\end{lemma}

\begin{proof}
Let us consider first $\frg=\frg_1$ with structure equations given as in Theorem~\ref{main-th}. Any $\rho\in Z^3(\frg)$ is given by
$$
\begin{array}{rl}
\rho &= a_1e^{125}+a_2e^{126}+a_3e^{135}+a_4e^{136}+a_6e^{156}+a_8e^{245}+a_9e^{246}\\[3pt]
 &\ \ +\,a_{10}e^{256}+a_{11}e^{345} + a_{12}e^{346}+a_{13}e^{356}+a_{14}e^{456},
\end{array}
$$
where $a_i\in\R$. We use the equation (\ref{def-J}) to compute the space of almost complex structures corresponding to $\rho\in Z^3(\frg)$.
When we compute the images of $e^5,e^6$ by $J_\rho^*$ we find that the subspace spanned by $e^5,e^6$ is $J_\rho^*$-invariant, because
$$
\begin{array}{rl}
J_{\rho}^* e^5 &= \frac{1}{\sqrt{|\lambda(\rho)|}}((a_1 a_{12}+a_{11} a_2-a_4 a_8-a_3 a_9)e^5+2 (a_{12} a_2-a_4 a_9)e^6),\\[9pt]
J_{\rho}^* e^6 &= \frac{1}{\sqrt{|\lambda(\rho)|}}(-2(a_1 a_{11}-a_3 a_8)e^5-(a_1 a_{12}+a_{11} a_2-a_4 a_8-a_3 a_9)e^6).
\end{array}
$$
Therefore, the (1,0)-form $\eta=e^5-iJ_\rho^*e^5$ is closed for every $\rho\in Z^3(\frg)$.

The same situation appears for $\frg=\frg_2^{\alpha}$, $\alpha\geq 0$, because again the subspace spanned by $e^5,e^6$
is found to be $J_\rho^*$-invariant for all $\rho\in Z^3(\frg)$.
\end{proof}

\begin{lemma}\label{reduc-prod}
Let $J$ be any complex structure on $\frg_1$ or $\frg_2^{\alpha}$, $\alpha\geq 0$,
with closed volume $(3,0)$-form.
Then, there is a $(1,0)$-basis $\{\omega^1,\omega^2,\omega^3\}$ satisfying the following reduced equations
\begin{equation}\label{equations2-1}
\left\{
\begin{array}{l}
d\omega^1=A\, \omega^{1}\wedge (\omega^{3}+\omega^{\bar{3}}),\\[2pt]
d\omega^2=-A\, \omega^{2}\wedge (\omega^{3}+\omega^{\bar{3}}),\\[2pt]
d\omega^3=0,
\end{array}
\right.
\end{equation}
where $A=\cos \theta+i\sin \theta$, $\theta\in[0,\pi)$.
\end{lemma}

\begin{proof}
By Lemma~\ref{1-forma-cerrada} we can consider a basis
of $(1,0)$-forms $\{\eta^1,\eta^2,\eta^3\}$ such that $\eta^3=e^5-iJ^*e^5$ is closed.
The structure equations of $\frg_1$ and $\frg_2^{\alpha}$ with $\alpha\geq 0$
force the differential of any 1-form to be a multiple of $e^5=\frac12(\eta^3+\eta^{\bar{3}})$,
so there exist $A,B,C,D,E,F\in \mathbb{C}$ such that
$$
\left\{
\begin{array}{l}
d\eta^1=(A\, \eta^{1}+B\, \eta^{2}+E\, \eta^{3})\wedge (\eta^{3}+\eta^{\bar{3}}),\\[7pt]
d\eta^2=(C\, \eta^{1}+D\, \eta^{2}+F\, \eta^{3})\wedge (\eta^{3}+\eta^{\bar{3}}),\\[7pt]
d\eta^3=0.
\end{array}
\right.
$$
Moreover, since $d(\eta^{123})=0$ necessarily $D=-A$.

Let us consider the non-zero 1-form $\tau^1=A\, \eta^{1}+B\, \eta^{2}+E\, \eta^{3}$.
Notice that
$$
d\tau^1=((A^2+BC)\eta^1+(AE+BF)\eta^3)\wedge (\eta^{3}+\eta^{\bar{3}}),$$
which implies that
$A^2+BC\not=0$ because otherwise $d\tau^1$ would be a multiple of $e^{56}$.
Then, with respect to the new (1,0)-basis $\{\tau^1,\tau^2,\tau^3\}$ given by
$$
\tau^1=A\, \eta^{1}+B\, \eta^{2}+E\, \eta^{3},\quad \tau^1=C\, \eta^{1}-A\, \eta^{2}+F\, \eta^{3},\quad \tau^3=\eta^{3},
$$
the complex structure equations are
\begin{equation}\label{redu}
\left\{
\begin{array}{l}
d\tau^1=(A\, \tau^{1}+B\, \tau^{2})\wedge (\tau^{3}+\tau^{\bar{3}}),\\[6pt]
d\tau^2=(C\, \tau^{1}-A\, \tau^{2})\wedge (\tau^{3}+\tau^{\bar{3}}),\\[6pt]
d\tau^3=0.
\end{array}
\right.
\end{equation}
Now we distinguish two cases:\\
$\bullet$ If $B\not=0$ then we consider the new basis $\{\omega^1, \omega^2, \omega^3\}$ given by
$$
\begin{array}{rl}
&\omega^1=\left(A+\sqrt{A^2+BC}\right)\tau^1+B\,\tau^2,\quad\  \omega^2=\left(A-\sqrt{A^2+BC}\right)\tau^1+B\,\tau^2,\\[8pt]
&\omega^3=\left|\sqrt{A^2+BC}\right| \tau^3.
\end{array}
$$
With respect to this basis, the equations \eqref{redu} reduce to
$$
d\omega^1=\frac{\sqrt{A^2+BC}}{\left|\sqrt{A^2+BC}\right|}\, \omega^{1}\wedge (\omega^{3}+\omega^{\bar{3}}),\ \
d\omega^2=-\frac{\sqrt{A^2+BC}}{\left|\sqrt{A^2+BC}\right|}\, \omega^{2}\wedge (\omega^{3}+\omega^{\bar{3}}),\ \
d\omega^3=0,
$$
that is, the equations are of the form \eqref{equations2-1} where the new complex coefficient has modulus equal to 1.\\
$\bullet$ If $C\not=0$ then with respect to the basis $\{\omega^1, \omega^2, \omega^3\}$ given by
$$
\begin{array}{rl}
&\omega^1=C\,\tau^1-\left(A-\sqrt{A^2+BC}\right)\tau^2,\quad\  \omega^2=C\,\tau^1-\left(A+\sqrt{A^2+BC}\right)\tau^2,\\[8pt]
&\omega^3=\left|\sqrt{A^2+BC}\right| \tau^3,
\end{array}
$$
the equations \eqref{redu} again reduce to equations of the form \eqref{equations2-1} where the new complex coefficient has modulus 1.

Finally, notice that in the equations \eqref{equations2-1} one can change the sign of $A$ by changing the
sign of $\omega^3$, so we can suppose that $A=\cos \theta+i\sin \theta$ with angle $\theta\in[0,\pi)$.
\end{proof}

\begin{proposition}\label{complex-moduli-g1-g2}
Up to isomorphism, there is only one complex structure with closed $(3,0)$-form on the Lie algebras $\frg_1$ and $\frg_2^{0}$,
whereas $\frg_2^{\alpha}$ has two such complex structures for any $\alpha>0$. More concretely, the complex
structures are:
\begin{align}
&(\frg_1,J) \colon \
d\omega^1=\omega^{1}\wedge (\omega^{3}+\omega^{\bar{3}}),\ \
d\omega^2=-\omega^{2}\wedge (\omega^{3}+\omega^{\bar{3}}),\ \
d\omega^3=0;\label{equations2-1-g1}\\[5pt]
&(\frg_2^0,J) \colon \
d\omega^1=i\,\omega^{1}\wedge (\omega^{3}+\omega^{\bar{3}}),\ \
d\omega^2=-i\,\omega^{2}\wedge (\omega^{3}+\omega^{\bar{3}}),\ \
d\omega^3=0;\label{equations2-1-g2-0}\\[5pt]
&(\frg_2^{\alpha=\frac{\cos \theta}{\sin \theta}},J_{\pm}) \colon \!
\left\{
\begin{array}{l}
d\omega^1= (\pm\cos \theta+i\sin \theta)\, \omega^{1}\wedge (\omega^{3}+\omega^{\bar{3}}),\\[2pt]
d\omega^2=- (\pm\cos \theta+i\sin \theta)\, \omega^{2}\wedge (\omega^{3}+\omega^{\bar{3}}),\\[2pt]
d\omega^3=0,
\end{array}
\right.\label{equations2-1-g2-alpha}\\[5pt]
&\hskip .2cm \mbox{where $\theta\in \left(0,\pi/2\right)$.}\nonumber
\end{align}
\end{proposition}

\begin{proof}
A real Lie algebra underlying the equations~\eqref{equations2-1} is isomorphic to $\frg_1$ or
$\frg_2^{\alpha}$ for some $\alpha\geq 0$. In fact, in terms of the real basis
$\beta^1,\ldots,\beta^6$ given by $\omega^1=\beta^1+i\beta^2$, $\omega^2=\beta^3+i\beta^4$
and $\omega^3=\frac12(\beta^5+i\beta^6)$, we have
$$
\begin{array}{lll}
d\beta^1=\cos \theta\, \beta^{15}-\sin \theta\, \beta^{25},\quad & d\beta^3=-\cos \theta\, \beta^{35}+\sin \theta\, \beta^{45},
\quad & d\beta^5=0,\\[5pt]
d\beta^2=\sin \theta\, \beta^{15}+\cos \theta\, \beta^{25},\quad & d\beta^4=-\sin \theta\, \beta^{35}-\cos \theta\, \beta^{45},
\quad & d\beta^6=0.
\end{array}
$$
In particular:\\
$\bullet$ If $\theta=0$ then taking $e^1=\beta^1$, $e^2=\beta^4$, $e^3=\beta^3$, $e^4=\beta^2$, $e^5=\beta^5$
and $e^6=\beta^6$ the resulting structure equations are precisely those of the Lie algebra $\frg_1$.\\
$\bullet$ If $\theta\in (0,\pi)$ then $\sin\theta\not=0$ and taking $e^1=\beta^1$, $e^2=-\beta^2$, $e^3=\beta^3$,
$e^4=\beta^4$, $e^5=\sin\theta\, \beta^5$
and $e^6=\beta^6$ we get the structure equations of $\frg_2^{|\alpha|}$ with $\alpha=-\frac{\cos \theta}{\sin \theta}$.
Notice that $\alpha$ takes any real value when $\theta$ varies in $(0,\pi)$, and if $\theta\not=\frac{\pi}{2}$
then $\theta$ and $\pi-\theta$ correspond to two complex structures on the same Lie algebra.
By a standard argument one can prove that these two complex structures are non-equivalent.
\end{proof}

\bigskip


Let us consider now the $3\oplus3$ decomposable Lie algebra $\frg_3$. With respect to the structure equations given in Theorem~\ref{main-th},
any closed 3-form $\rho\in Z^3(\frg_3)$ is given by
$$
\begin{array}{rl}
\rho\!\! &=\, a_1\,e^{123}+a_2\,e^{124}+a_3\,e^{134}+a_4\,e^{145}+a_5\,e^{146}+a_6\,e^{156} +a_7\,e^{234}\\[3pt]
 &\ \ +\,a_8(e^{136}-e^{245})+a_9(e^{135}-e^{246})+a_{10}(e^{126}+e^{345})\\[3pt]
 &\ \ +\,a_{11}(e^{125}+e^{346})+a_{12}\,e^{456},
\end{array}
$$
where $a_1,\ldots,a_{12} \in \mathbb{R}$.
By imposing the closedness of $J_\rho^* \rho$ together with the condition tr$(J_\rho^{*2})<0$,
one arrives by a long computation to an explicit description of the complex structure $J_{\rho}$, which
allows us to prove that $\{e^1,e^2,e^3,J_{\rho}^* e^1,J_{\rho}^* e^2,J_{\rho}^* e^3\}$
are always linearly independent. Therefore, the forms
$$
\omega^1=e^1-i J_{\rho}^* e^1,\quad \omega^2=e^2-i J_{\rho}^* e^2,\quad \omega^3=e^3-i J_{\rho}^* e^3,
$$
constitute a (1,0)-basis for the complex structure $J_{\rho}$.
Moreover, one can show that with respect to this basis
the complex structure equations have the form
\begin{equation}\label{equations_e2+e11}
\left\{
\begin{array}{l}
d\omega^1=0,\\[4pt]
d\omega^2=-\frac12 \omega^{13} +b\, \omega^{1\bar{1}}+fi\,\omega^{1\bar{2}}-fi\,\omega^{2\bar{1}}-(\frac12+gi)\omega^{1\bar{3}}+gi\,\omega^{3\bar{1}},\\[5pt]
d\omega^3=\frac12 \omega^{12} +c\, \omega^{1\bar{1}}+(\frac12+hi)\omega^{1\bar{2}}
-hi\,\omega^{2\bar{1}}-fi\,\omega^{1\bar{3}}+fi\,\omega^{3\bar{1}},
\end{array}
\right.
\end{equation}
where the coefficients $b,c,f,g,h$ are real and satisfy $4gh=4f^2-1$. (Explicit expression of each coefficient in terms of
$a_1,\ldots,a_{12}$ can be given, but this information is not relevant and so we omit it.) Notice that the condition $4gh=4f^2-1$
is equivalent to the Jacobi identity $d^2=0$. Furthermore, these equations
can be reduced as follows:

\begin{proposition}\label{complex-moduli-g3}
Up to isomorphism, the complex structures with closed $(3,0)$-form on the Lie algebra
$\frg_3$ are
\begin{equation}\label{reduced_equations_e2+e11}
(\frg_3,J_x)\colon \, \left\{
\begin{array}{l}
d\omega^1=0,\\[4pt]
d\omega^2=-\frac12 \omega^{13} -(\frac12+xi)\omega^{1\bar{3}}+xi\,\omega^{3\bar{1}},\\[5pt]
d\omega^3=\frac12 \omega^{12} +(\frac12-\frac{i}{4x})\omega^{1\bar{2}}
+\frac{i}{4x}\,\omega^{2\bar{1}},
\end{array}
\right.
\end{equation}
where $x\in \mathbb{R}^+$.
\end{proposition}

\begin{proof}
Observe first that with respect to the (1,0)-basis $\{ \omega^1, \omega^2+2c\,\omega^1, \omega^3-2b\,\omega^1 \}$,
the complex structure
equations express again as in (\ref{equations_e2+e11}) but with $b=c=0$, that is to say, one can suppose that the coefficients
$b$ and $c$ both vanish.

Let us prove next that we can also take the coefficient $f$ to be zero. To see this, let $\{\omega^1,\omega^2,\omega^3\}$
be a (1,0)-basis satisfying (\ref{equations_e2+e11}) with $b=c=0$ and $f\neq0$, and let us consider the (1,0)-basis
$\{\eta^1,\eta^2,\eta^3\}$ given by

\medskip

$
\eta^1=\omega^1,\qquad
\eta^2=\omega^{2}-\frac{g-h-\sqrt{1+(g+h)^2}}{2f}\,\omega^{3},\qquad
\eta^3=\frac{g-h-\sqrt{1+(g+h)^2}}{2f}\omega^{2}+\omega^3.
$

\medskip

\noindent A direct calculation shows that with respect to $\{\eta^1,\eta^2,\eta^3\}$ the corresponding coefficient $f$ vanishes.
Therefore, since $4gh=-1$ we are led to the reduced equations \eqref{reduced_equations_e2+e11},
where we have written $x$ instead of $g$.

Finally, let $J_x$ and $J_{x'}$ be two complex structures corresponding to $x,x'\in \mathbb{R}$.
It is easy to see that the
structures are equivalent if and only if $xx'=-\frac14$. This represents an hyperbola in the $(x,x')$-plane, so
the equivalence class is given by one of the branches of the hyperbola, that is, we can take $x>0$.
\end{proof}

\subsection{The indecomposable case}

Next we classify the complex structures on the indecomposable Lie algebras
$\frg_4,\ldots,\frg_9$.

\begin{lemma}\label{reduc-noprod}
Let $J$ be any complex structure on $\frg_k$ $(4\leq k\leq 7)$
with closed $(3,0)$-form.
Then, there is a $(1,0)$-basis $\{\omega^1,\omega^2,\omega^3\}$ such that
\begin{equation}\label{equations2-2}
\left\{
\begin{array}{l}
d\omega^1=A\,\omega^{1}\wedge (\omega^{3}+\omega^{\bar{3}}),\\[2pt]
d\omega^2=-A\,\omega^{2}\wedge (\omega^{3}+\omega^{\bar{3}}),\\[2pt]
d\omega^3=G_{11}\,\omega^{1\bar{1}}+G_{12}\,\omega^{1\bar{2}}+
\overline{G}_{12}\,\omega^{2\bar{1}}+G_{22}\,\omega^{2\bar{2}},
\end{array}
\right.
\end{equation}
where $A,G_{12}\in \mathbb{C}$ and $G_{11},G_{22}\in \mathbb{R}$, with $(G_{11},G_{12},G_{22})\not=(0,0,0)$, satisfy
\begin{equation}\label{equations2-2-cond}
|A|=1,\quad (A+\overline{A})G_{11}=0, \quad (A+\overline{A})G_{22}=0,\quad (A-\overline{A})G_{12}=0.
\end{equation}
\end{lemma}

\begin{proof}
Let us consider first the Lie algebra $\frg_4$ with structure equations given as in Theorem~\ref{main-th}.
Any element $\rho\in Z^3(\frg_4)$ is given by
$$
\begin{array}{l}
\rho=a_1\,e^{123}+a_2\,e^{126}+a_3\,(e^{125}-e^{134})+a_4\,(e^{124}+e^{135})+a_5\,e^{136}\\[4pt]
\phantom{iii}+a_6\,(e^{156}+e^{234})+a_7\,(e^{146}-e^{235})+a_8\,e^{236}+a_9\,e^{246}+a_{10}\,e^{256}\\[4pt]
\phantom{iii}+a_{11}\,e^{346}+a_{12}\,e^{356}+a_{13}\,e^{456},
\end{array}
$$
where $a_1,\ldots,a_{13}\in \mathbb{R}$.
A direct calculation shows that if $a_3^2+a_4^2=0$ then there do not exist closed 3-forms $\rho$ satisfying the conditions
$d(J_\rho^*\rho)=0$ and $\lambda(\rho)<0$.

Suppose that $a_3^2+a_4^2\not=0$. Then, an element $\rho\in Z^3(\frg_4)$ satisfies the condition
$d(J_\rho^*\rho)=0$ if and only if
$a_{10}=\frac{a_3(a_6^2-a_7^2) + 2a_4a_6a_7 - a_{11}(a_3^2+a_4^2)}{a_3^2+a_4^2}$,
$a_{12}=\frac{2a_3a_6a_7 - a_4(a_6^2 - a_7^2) + a_9(a_3^2 + a_4^2)}{a_3^2+a_4^2}$
and $a_{13}=0$. Moreover, under these relations one has that $\lambda(\rho)=-4 (a_3 a_9-a_4a_{11}+a_6 a_7)^2\leq 0$.

Let $\rho\in Z^3(\frg_4)$ be such that $\lambda(\rho)<0$ and $d(J_\rho^*\rho)=0$. A direct calculation shows that
$\tilde{J}_\rho^*e^6$ is given by
$$
\begin{array}{l}
\tilde{J}_\rho^*e^6=2 (a_3^2+a_4^2)e^1+2 (a_3 a_6+a_4 a_7)e^2+2 (a_3 a_7-a_4 a_6) e^3\\[4pt]
\phantom{\tilde{J}_\rho^*e^6iii}+2 (a_3a_{11}+a_4 a_9+a_7^2) e^6.
\end{array}
$$
Therefore, the coefficient of $J_\rho^*e^6$ in $e^1$ is nonzero for any $\rho$.

A similar computation for the Lie algebras $\frg_5$, $\frg_6$ and $\frg_7$ shows that for any complex structure $J_\rho$
with closed $(3,0)$-form, we also have that
$$
J_\rho^*e^6= c_1 e^1+c_2 e^2+c_3 e^3+c_4 e^4+c_5 e^5+c_6 e^6,
$$
where the coefficient $c_1$ is non-zero.

Let us consider the $(1,0)$-form $\eta^3=e^6-i J_{\rho}^* e^6$. From the structure equations
of $\frg_k$ $(4\leq k\leq 7)$ in Theorem~\ref{main-th}, it follows that
$$
\begin{array}{lll}
d\eta^3=i c_1 e^{23} - i \alpha\wedge e^6, &\text{if}&\frg=\frg_4,\\[6pt]
d\eta^3=i c_1 (e^{24}+e^{35}) - i \alpha\wedge e^6, &\text{if}&\frg=\frg_5, \frg_6, \frg_7,
\end{array}
$$
where $\alpha$ is a 1-form. Since $c_1\not=0$ we can write the 2-forms $e^{23}$ and $e^{24}+e^{35}$ as
\begin{equation}\label{expresion}
\begin{array}{lll}
e^{23}= -\frac{i}{c_1} d\eta^3+\frac{1}{c_1}\alpha\wedge e^6, &\text{if}&\frg=\frg_4,\\[6pt]
e^{24}+e^{35}= -\frac{i}{c_1} d\eta^3+\frac{1}{c_1}\alpha\wedge e^6, &\text{if}&\frg=\frg_5, \frg_6, \frg_7.
\end{array}
\end{equation}
Now, let $\eta^1,\eta^2$ be such that $\{\eta^1,\eta^2,\eta^3\}$ is a basis of $(1,0)$-forms.
Since $e^6$ is closed and $\eta^3+\eta^{\bar{3}}=2e^6$, the integrability
of the complex structure implies that $d\eta^3$ has no component of type $(2,0)$ and
$$
d\eta^3=G_{11}\,\eta^{1\bar{1}}+G_{12}\,\eta^{1\bar{2}}+G_{13}\,\eta^{1\bar{3}}+
\overline{G}_{12}\,\eta^{2\bar{1}}+G_{22}\,\eta^{2\bar{2}}+G_{23}\,\eta^{2\bar{3}}+
\overline{G}_{13}\,\eta^{3\bar{1}}+\overline{G}_{23}\,\eta^{3\bar{2}}+G_{33}\,\eta^{3\bar{3}},
$$
for some $G_{11},G_{22},G_{33}\in \mathbb{R}$ and $G_{12},G_{13},G_{23}\in \mathbb{C}$.

From the structure of the Lie algebras $\frg_k$ $(4\leq k\leq 7)$,
the relation \eqref{expresion} and taking into account that $d\eta^3$ is of type (1,1),
it follows that there exist $\lambda,\mu\in \mathbb{C}$ such that
\begin{equation}\label{ecus-noprod1}
\left\{
\begin{array}{l}
d\eta^1=\lambda\,d\eta^3 +(A\, \eta^{1}+B\, \eta^{2}+E\, \eta^{3})\wedge (\eta^{3}+\eta^{\bar{3}}),\\[7pt]
d\eta^2=\mu\,d\eta^3 +(C\, \eta^{1}+D\, \eta^{2}+F\, \eta^{3})\wedge (\eta^{3}+\eta^{\bar{3}}),\\[7pt]
d\eta^3=\ \, G_{11}\,\eta^{1\bar{1}}+G_{12}\,\eta^{1\bar{2}}+G_{13}\,\eta^{1\bar{3}}\\[5pt]
\phantom{d\eta^3ii}+\overline{G}_{12}\,\eta^{2\bar{1}}+G_{22}\,\eta^{2\bar{2}}+G_{23}\,\eta^{2\bar{3}}\\[5pt]
\phantom{d\eta^3ii}+\overline{G}_{13}\,\eta^{3\bar{1}}+\overline{G}_{23}\,\eta^{3\bar{2}}+G_{33}\,\eta^{3\bar{3}},
\end{array}
\right.
\end{equation}
for some $A,B,C,D,E,F\in \mathbb{C}$.

We will see next that these complex equations can be reduced to equations of the form~\eqref{equations2-2}.
Notice first that with respect to the (1,0)-basis $\{\eta^1-\lambda\,\eta^3, \eta^2-\mu\,\eta^3, \eta^3\}$
we get complex equations of the form~\eqref{ecus-noprod1} with $\lambda=\mu=0$.
So, without loss of generality we can suppose $\lambda=\mu=0$.
Moreover, the coefficients $E$ and $F$ also vanish. In fact, suppose for example that $E\not=0$ (the case $F\not=0$ is similar).
Using~\eqref{ecus-noprod1} with $\lambda=\mu=0$, the condition $d(d\eta^1)=0$ is equivalent to
$$
E G_{11}= E G_{12}= E G_{13}= E G_{22}= E G_{23}= 0,
$$
so $E\not=0$ implies $d\eta^3=G_{33}\,\eta^{3\bar{3}}=G_{33}\,\eta^{3}\wedge (\eta^{3}+\eta^{\bar{3}})$.
But this is a contradiction with the structure of the Lie algebras $\frg_k$ $(4\leq k\leq 7)$, because $d(\frg_k^*)$ would be annihilated by
the real 1-form $\eta^{3}+\eta^{\bar{3}}$.

From now on, we suppose that $\lambda=\mu=E=F=0$ in the equations~\eqref{ecus-noprod1}.
A direct calculation shows that
$$
d \eta^{123}=\overline{G}_{13}\,\eta^{123\bar{1}}+\overline{G}_{23}\,\eta^{123\bar{2}}+(A+D+G_{33})\eta^{123\bar{3}},
$$
so $\eta^{123}$ is closed if and only if $G_{13}=G_{23}=0$ and $D=-A-G_{33}$.
Moreover, the unimodularity of the Lie algebras $\frg_k$ $(4\leq k\leq 7)$ implies that $G_{33}=0$. In fact, taking
the real basis $\{f^1,\ldots,f^6\}$ of $\frg_k^*$ given by
$$
\eta^1=f^2+if^3, \quad\quad  \eta^2=f^4+if^5, \quad\quad  \eta^3=f^6+if^1,
$$
we get that the trace of ${\rm ad}_{f_6}$
is zero if and only if $G_{33}=-2\,\Real A-2\,\Real D$, which implies, using that $G_{33}=-A-D$,
that the coefficient $G_{33}=0$.

Summing up, we have proved the existence of a $(1,0)$-basis $\{\eta^1,\eta^2,\eta^3\}$ satisfying the reduced complex equations
\begin{equation}\label{ecus-noprod2}
\left\{
\begin{array}{l}
d\eta^1=(A\, \eta^{1}+B\, \eta^{2})\wedge (\eta^{3}+\eta^{\bar{3}}),\\[7pt]
d\eta^2=(C\, \eta^{1}-A\, \eta^{2})\wedge (\eta^{3}+\eta^{\bar{3}}),\\[7pt]
d\eta^3=\ \, G_{11}\,\eta^{1\bar{1}}+G_{12}\,\eta^{1\bar{2}}+\overline{G}_{12}\,\eta^{2\bar{1}}+G_{22}\,\eta^{2\bar{2}},
\end{array}
\right.
\end{equation}
where $A,B,C,G_{12}\in \mathbb{C}$ and $G_{11},G_{22}\in \mathbb{R}$.

Notice that $A^2+BC\not=0$ because otherwise the (1,0)-form $A\, \eta^{1}+B\, \eta^{2}$ would be closed, but this is
a contradiction to $b_1(\frg_k)=1$, for $4\leq k\leq 7$. Therefore, arguing as in the proof of
Lemma~\ref{reduc-prod} we can suppose that $B=C=0$ and $|A|=1$ in \eqref{ecus-noprod2}.
Finally, the condition $d(d\eta^3)=0$ is satisfied if and only if $(A+\overline{A})G_{11}=(A+\overline{A})G_{22}=(A-\overline{A})G_{12}=0$.
\end{proof}

As a consequence of the previous lemma, we have the following classification of complex structures on $\frg_k$, for $4\leq k\leq 7$.

\begin{proposition}\label{complex-moduli-g4-g5-g6-g7}
Up to isomorphism there is only one complex structure $J$ with closed $(3,0)$-form
on the Lie algebras $\frg_5$ and $\frg_6$,
and two such complex structures on the Lie algebras $\frg_4$ and $\frg_7$.
More concretely, the complex structures are:
\begin{align}
&(\frg_4,J_{\pm})\colon
d\omega^1\!=i\,\omega^{1}\!\wedge (\omega^{3}+\omega^{\bar{3}}),\
d\omega^2\!=\! -i\,\omega^{2}\!\wedge (\omega^{3}+\omega^{\bar{3}}),\
d\omega^3\!=\pm\,\omega^{1\bar{1}};\label{equations2-2-g4}\\[5pt]
&(\frg_5,J)\colon
d\omega^1\!=\omega^{1}\!\wedge (\omega^{3}+\omega^{\bar{3}}),\
d\omega^2\!=\! -\omega^{2}\!\wedge (\omega^{3}+\omega^{\bar{3}}),\
d\omega^3\!=\omega^{1\bar{2}}+\omega^{2\bar{1}};\label{equations2-2-g5}\\[5pt]
&(\frg_6,J) \colon
d\omega^1\!=i\,\omega^{1}\!\wedge (\omega^{3}+\omega^{\bar{3}}),\
d\omega^2\!=\! -i\,\omega^{2}\!\wedge (\omega^{3}+\omega^{\bar{3}}),\
d\omega^3\!=\omega^{1\bar{1}}+\omega^{2\bar{2}};\label{equations2-2-g6}\\[5pt]
&(\frg_7,J_{\pm}) \colon
d\omega^1\!=i\,\omega^{1}\!\wedge\! (\omega^{3}\!\!+\omega^{\bar{3}}),\
d\omega^2\!=\! -i\,\omega^{2}\!\wedge\! (\omega^{3}\!\!+\omega^{\bar{3}}),\
d\omega^3\!=\pm(\omega^{1\bar{1}}\!\!-\omega^{2\bar{2}}).\label{equations2-2-g7}
\end{align}
\end{proposition}

\begin{proof}
First notice that in the equations \eqref{equations2-2}, after changing the
sign of $\omega^3$ if necessary, we can always suppose that $A=\cos \theta+i\sin \theta$ with angle $\theta\in[0,\pi)$.
We have the following cases:

\smallskip

\noindent $\bullet$ If $\cos \theta\not=0$, then \eqref{equations2-2-cond} implies $G_{11}=G_{22}=0$ and $\sin \theta\, G_{12}=0$,
so $\sin \theta=0$ because $(G_{11},G_{12},G_{22})\not=(0,0,0)$ is satisfied if and only if $G_{12}\not=0$.
Therefore, in this case $A=1$ and, moreover, we can normalize the coefficient $G_{12}$
(it suffices to consider $G_{12}\,\omega^1$ instead of $\omega^1$).
So the complex structure equations take the form \eqref{equations2-2-g5}, and
in terms of the real basis $\{e^1,\ldots,e^6\}$ defined by $\omega^1=e^2-ie^3$, $\omega^2=e^5+ie^4$
and $\omega^3=\frac12e^6-2ie^1$, one has
$$
de^1=e^{24}+e^{35},\ \
de^2=e^{26},\ \
de^3=e^{36},\ \
de^4=-e^{46},\ \
de^5=-e^{56},\ \
de^6=0,
$$
that is, the underlying Lie algebra is $\frg_5$.

\smallskip

\noindent $\bullet$ If $\cos \theta=0$, then \eqref{equations2-2-cond} implies that $A=i$ and $G_{12}=0$.
Therefore, the complex structure equations become
$$
d\omega^1=i\,\omega^{1}\wedge (\omega^{3}+\omega^{\bar{3}}),\quad
d\omega^2=-i\,\omega^{2}\wedge (\omega^{3}+\omega^{\bar{3}}),\quad
d\omega^3=G_{11}\,\omega^{1\bar{1}}+G_{22}\,\omega^{2\bar{2}},
$$
where $(G_{11},G_{22})\not=(0,0)$. We have the following possibilities:

\smallskip

 {\bf -} When $G_{22}=0$ we can suppose that $G_{11}=\pm 1$
(it suffices to consider $\sqrt{|G_{11}|}\,\omega^1$ instead of $\omega^1$), and then
the complex structure equations reduce to \eqref{equations2-2-g4}.
In terms of the real basis $\{e^1,\ldots,e^6\}$ given by $\omega^1=e^2-ie^3$, $\omega^2=e^4+ie^5$
and $\omega^3=-\frac12e^6\pm 2ie^1$, we arrive at
$$
de^1=e^{23},\ \
de^2=-e^{36},\ \
de^3=e^{26},\ \
de^4=-e^{56},\ \
de^5=e^{46},\ \
de^6=0,
$$
that is, the underlying Lie algebra is $\frg_4$. A standard argument
allows to conclude that the two complex structures in \eqref{equations2-2-g4}
are non-isomorphic.

\smallskip

 {\bf -} The case $G_{11}=0$ easily reduces to the previous case, so it does not produce any non-isomorphic complex structure.

\smallskip

 {\bf -} Finally, if $G_{11}\not=0$ and $G_{22}\not=0$ then we can suppose $G_{11}=\pm 1$ and $G_{22}=\pm 1$
(it suffices to consider $\sqrt{|G_{kk}|}\,\omega^k$ instead of $\omega^k$ for $k=1,2$).
It is clear that the case $G_{11}=G_{22}=-1$ is equivalent to $G_{11}=G_{22}=1$, so it remains to study
the following three cases: $(G_{11},G_{22})=(1,1),(1,-1),(-1,1)$.
In terms of the real basis $\{\beta^1,\ldots,\beta^6\}$ defined by $\omega^1=\beta^2+i\beta^4$, $\omega^2=\beta^3+i\beta^5$
and $\omega^3=\frac12\beta^6+ 2i\beta^1$, one has
$$
d\beta^1=-G_{11}\,\beta^{24}-G_{22}\,\beta^{35},\
d\beta^2=-\beta^{46},\
d\beta^3=\beta^{56},\
d\beta^4=\beta^{26},\
d\beta^5=-\beta^{36},\
d\beta^6=0.
$$

When $(G_{11},G_{22})=(1,1)$,
taking the basis $e^1=-2 \beta^1$, $e^2=\beta^2+\beta^3$, $e^3=-\beta^4+\beta^5$, $e^4=\beta^4+\beta^5$, $e^5=\beta^2-\beta^3$ and $e^6=-\beta^6$,
the real structure equations are
$$
de^1=e^{24}+e^{35},\
de^2=-e^{36},\
de^3=e^{26},\
de^4=-e^{56},\
de^5=e^{46},\
de^6=0,
$$
so the underlying Lie algebra is $\frg_6$ and the complex structure is given by \eqref{equations2-2-g6}.

The cases $(G_{11},G_{22})=(1,-1)$ and $(G_{11},G_{22})=(-1,1)$ both correspond to the same Lie algebra
(in fact, a change in the sign of~$\beta^1$ gives an isomorphism), so we suppose next that $(G_{11},G_{22})=(1,-1)$,
i.e.
$$
d\beta^1=-\beta^{24}+\beta^{35},\
d\beta^2=-\beta^{46},\
d\beta^3=\beta^{56},\
d\beta^4=\beta^{26},\
d\beta^5=-\beta^{36},\
d\beta^6=0.
$$
Taking $e^1=-\beta^1$, $e^3=-\beta^3$ and $e^6=-\beta^6$, we conclude that $\frg_7$ is
the underlying Lie algebra. Therefore, the complex structures on $\frg_7$ are given by \eqref{equations2-2-g4},
and it can be proved that they are non-isomorphic.
\end{proof}

Next we find that there are infinitely many non-isomorphic complex structures on the Lie algebra $\frg_8$.

\begin{proposition}\label{complex-moduli-g8}
Let $J$ be any complex structure on $\frg_8$ with closed volume $(3,0)$-form.
Then, there is a $(1,0)$-basis $\{\omega^1,\omega^2,\omega^3\}$ satisfying one of the following reduced equations:
\begin{align}
&(\frg_8,J) \colon \
d\omega^1= 2i\,\omega^{13} +\omega^{3\bar{3}},\ \,
d\omega^2= -2i\,\omega^{23},\ \,
d\omega^3=0;\label{equations2-3-1}\\[5pt]
&(\frg_8,J') \colon \
d\omega^1=2i\,\omega^{13}+\omega^{3\bar{3}},\ \,
d\omega^2= -2i\,\omega^{23}+\omega^{3\bar{3}},\ \,
d\omega^3=0;\label{equations2-3-2}\\[5pt]
&(\frg_8,J_A) \colon \! \left\{
\begin{array}{l}
d\omega^1=-(A-i)\omega^{13}-(A+i)\omega^{1\bar{3}},\\[2pt]
d\omega^2=(A-i)\omega^{23}+(A+i)\omega^{2\bar{3}},\\[2pt]
d\omega^3=0,\label{equations2-3-3}
\end{array}
\right.\\[5pt]
&\hskip .2cm \mbox{where $A\in \mathbb{C}$ with $\Imag A\not=0$.}\nonumber
\end{align}

\noindent Moreover, the complex structures above are non-isomorphic.
\end{proposition}

\begin{proof}
With respect to the structure equations of $\frg_8$ given in Theorem~\ref{main-th},
any closed 3-form $\rho\in Z^3(\frg_8)$ is given by
$$
\begin{array}{l}
\rho=a_1\,e^{126}+a_2\,e^{135}+a_3\,e^{145}+a_4\,e^{156}+a_5\,e^{235}+a_6(e^{146}+e^{236})\\[4pt]
\phantom{iii}+a_7\,e^{245}+a_8(e^{136}-e^{246})+a_9\,e^{256}+a_{10}\,e^{346}+a_{11}\,e^{356}+a_{12}\,e^{456},
\end{array}
$$
where $a_1,\ldots,a_{12}\in \mathbb{R}$.
A direct calculation shows that such a $\rho$ satisfies the conditions $d(J_\rho^*{\rho})=0$ and $\lambda(\rho)<0$
if and only if
$a_1=0$, $a_2=-a_7$, $a_3=a_5$, $a_{10}=0$ and $a_6a_7-a_5a_8\neq 0$. Moreover, in this case $\lambda(\rho)=-4 (a_6 a_7-a_5 a_8)^2$.

The associated complex structures $J_\rho^*$ express in terms of the real basis $\{e^1,\ldots,e^6\}$ as
$$
\begin{array}{l}
 J_{\rho}^* e^1= e^2+ \frac{a_5 a_{12} - a_7 a_{11}}{a_6 a_7 - a_5 a_8} e^5+ \frac{a_6 a_{12} - a_8 a_{11}}{a_6 a_7 - a_5 a_8} e^6,\\[7pt]
 J_{\rho}^* e^2= - e^1 +\frac{a_5 a_{11} + a_7 a_{12}}{a_6 a_7 - a_5 a_8}e^5 +\frac{a_6 a_{11} + a_8 a_{12}}{a_6 a_7 - a_5 a_8} e^6,\\[7pt]
 J_{\rho}^* e^3= e^4+\frac{a_4 a_7-a_5 a_9}{a_6 a_7 - a_5 a_8} e^5 + \frac{a_4 a_8 - a_6 a_9}{a_6 a_7 - a_5 a_8} e^6,\\[7pt]
 J_{\rho}^* e^4= - e^3-\frac{a_4 a_5 + a_7 a_9}{a_6 a_7 - a_5 a_8} e^5 -\frac{a_4 a_6 + a_8 a_9}{a_6 a_7 - a_5 a_8} e^6,\\[7pt]
 J_{\rho}^* e^5=\frac{a_5 a_6 +  a_7 a_8}{a_6 a_7 - a_5 a_8} e^5+ \frac{a_6^2 + a_8^2}{a_6 a_7 - a_5 a_8} e^6,\\[7pt]
 J_{\rho}^* e^6=-\frac{a_5^2 + a_7^2}{a_6 a_7 - a_5 a_8} e^5- \frac{a_5 a_6 + a_7 a_8}{a_6 a_7 - a_5 a_8} e^6.
\end{array}
$$
Let us consider the basis of (1,0)-forms $\{\omega^1,\omega^2,\omega^3\}$ given by
$$
\begin{array}{l}
\omega^1=e^1-i J_{\rho}^* e^1= e^1 - i \left( e^2+ k_1 e^5+ k_2 e^6\right),\\[7pt]
\omega^2=e^3-i J_{\rho}^* e^3= e^3 - i \left( e^4+k_3 e^5 + k_4 e^6\right),\\[7pt]
\omega^3=\frac{1}{2c}(e^5-i  J_{\rho}^* e^5)= \frac{1}{2c} e^5-i\left(\frac{b}{2c} e^5+ \frac{1}{2} e^6\right),
\end{array}
$$
where $k_1=\frac{a_5 a_{12} - a_7 a_{11}}{a_6 a_7 - a_5 a_8}$, $k_2=\frac{a_6 a_{12} - a_8 a_{11}}{a_6 a_7 - a_5 a_8}$,
$k_3=\frac{a_4 a_7-a_5 a_9}{a_6 a_7 - a_5 a_8}$, $k_4=\frac{a_4 a_8 - a_6 a_9}{a_6 a_7 - a_5 a_8}$,
$b=\frac{a_5 a_6 +  a_7 a_8}{a_6 a_7 - a_5 a_8}$ and $c=\frac{a_6^2 + a_8^2}{a_6 a_7 - a_5 a_8}$.
Notice that $c\not=0$, and $-2(a_6+i a_8)\omega^{123}=\rho+i\, J_{\rho}^* \rho$.

\smallskip

With respect to this basis, the complex structure equations are
\begin{equation}\label{equations_nakamura}
\left\{
\begin{array}{l}
d\omega^1=-(A-i)\omega^{13}-(A+i)\omega^{1\bar{3}}+B\,\omega^{3\bar{3}},\\[7pt]
d\omega^2=(A-i)\omega^{23}+(A+i)\omega^{2\bar{3}}+C\,\omega^{3\bar{3}},\\[7pt]
d\omega^3=0,
\end{array}
\right.
\end{equation}
where $A=b+ic$, $B=2c(k_1+ik_2)$ and $C=-2c(k_3+ik_4)$. Notice that $\Imag A=c\not=0$.

\noindent Now, we will reduce the complex equations \eqref{equations_nakamura} as follows:

\smallskip

\noindent $\bullet$ If $A\neq-i$, then with respect to the (1,0)-basis $\{\eta^1,\eta^2,\eta^3\}$ given by
$$
\eta^1=-(A+i)\omega^1+B\omega^3,\quad\eta^2=(A+i)\omega^2+C\omega^3,\quad\eta^3=\omega^3,
$$
the complex structure equations are of the form \eqref{equations2-3-3}.

\smallskip

\noindent $\bullet$ If $A=-i$, the equations \eqref{equations_nakamura} reduce to
$$
J_{(B,C)} \colon \, d\omega^1=2i\omega^{13}+B\omega^{3\bar{3}},\ \
d\omega^2=-2i\omega^{23}+C\omega^{3\bar{3}},
\ \ d\omega^3=0.
$$
Notice that the structures $J_{(B,C)}$ and $J_{(C,B)}$ are equivalent, since it suffices
to consider the change of basis $\eta^1=\omega^2$, $\eta^2=\omega^1$, $\eta^3=-\omega^3$.
Now:

\smallskip

 {\bf -} if $B=C=0$ then the complex equations are of the form \eqref{equations2-3-3} with $A=-i$;

\smallskip

 {\bf -} if only one of the coefficients $B,C$ is nonzero, for instance $B$,
then taking $\frac{1}{B}\omega^1$ instead of $\omega^1$, we arrive at the complex equations \eqref{equations2-3-1};

\smallskip

 {\bf -} finally, if $B,C\neq0$ then we can normalize both coefficients and
 the corresponding complex equations are \eqref{equations2-3-2}.

It is straightforward to check that the complex structures given in equations \eqref{equations2-3-1}--\eqref{equations2-3-3}
are non-isomorphic.
\end{proof}

\begin{remark}\label{remark-Nakamura}
{\rm
Note that on $\frg_8$ there exists a
unique complex structure $J$ that is abelian \cite{ABD},
i.e. satisfying $[JX,JY]=[X,Y]$,
which corresponds to the value $A=i$ in equations \eqref{equations2-3-3},
and a unique bi-invariant complex structure \cite{Nak}, corresponding to $A=-i$ in \eqref{equations2-3-3}.
}
\end{remark}


Finally, let us consider now the Lie algebra $\frg_9$.  With respect to the structure equations given in Theorem~\ref{main-th},
any closed 3-form $\rho\in Z^3(\frg_9)$ is given by
$$
\begin{array}{rl}
\rho &= a_1\,(e^{124}-e^{135})+a_2\,e^{145}+a_3\,e^{146}+a_4\,e^{156}+a_5(e^{136}-e^{245})\\[3pt]
 &\ \ +\,a_6(e^{125}+e^{134}- e^{246})+
a_7\,e^{256}+a_8(e^{126}+e^{345})+a_9\,e^{346}\\[3pt]
 &\ \ +\,a_{10}(e^{125}+e^{134}+e^{356})+a_{11}\,e^{456}.
\end{array}
$$
By imposing the closedness of $J_\rho^* \rho$ together with the condition tr$(J_\rho^{*2})<0$,
one can arrive after a long computation to an explicit description of the complex structure $J_{\rho}$, which
allows us to prove that $\{e^2,e^4,e^6,J_{\rho}^* e^2,J_{\rho}^* e^4,J_{\rho}^* e^6\}$
are always linearly independent.
Therefore, the forms
$$
\omega^1=e^6-i J_{\rho}^* e^6,\quad \omega^2=e^2-i J_{\rho}^* e^2,\quad \omega^3=e^4-i J_{\rho}^* e^4,
$$
constitute a (1,0)-basis for the complex structure $J_{\rho}$, and one can show that with respect to this basis
the complex structure equations have the form
\begin{equation}\label{equations_B64}
\left\{
\begin{array}{l}
d\omega^1=-c^2\omega^{1\bar{1}}-c\,\omega^{3\bar{1}}-c\,\omega^{1\bar{3}}-\omega^{3\bar{3}},\\[7pt]
d\omega^2=(\frac{c}{2}+cE-\frac{i}{2}G)\omega^{1\bar{1}}-\frac{i}{2}\omega^{2\bar{1}}+E\,\omega^{3\bar{1}}+(\frac12+cG) \omega^{1\bar{3}}\\[5pt]
\phantom{d\omega^2=}+G\,\omega^{3\bar{3}}+\frac{i}{2} \omega^{12} +(cG-E) \omega^{13},\\[7pt]
d\omega^3=c(c^2+\frac{i}{2})\omega^{1\bar{1}}+(c^2+\frac{i}{2})\omega^{3\bar{1}}+c^2\omega^{1\bar{3}}+c\,\omega^{3\bar{3}}-\frac{i}{2} \omega^{13},
\end{array}
\right.
\end{equation}
where $c$ is real and $E,G \in \mathbb{C}$.
(Explicit expression of each coefficient in terms of
$a_1,\ldots,a_{11}$ can be given, but this information will not be relevant in what follows and so we omit it.)
In the following result we prove that all the complex structures are equivalent.

\begin{proposition}\label{complex-moduli-g9}
Up to isomorphism, there is only one complex structure with closed $(3,0)$-form on the Lie algebra
$\frg_9$, whose complex equations are
\begin{equation}\label{reduced_equations_B64}
(\frg_9,J) \colon \,
d\omega^1\!=\!-\omega^{3\bar{3}},\ \
d\omega^2\!=\frac{i}{2} \omega^{12}+\frac12 \omega^{1\bar{3}}-\frac{i}{2}\omega^{2\bar{1}},\ \
d\omega^3\!=\!-\frac{i}{2} \omega^{13}+\frac{i}{2}\omega^{3\bar{1}}.
\end{equation}
\end{proposition}

\begin{proof}
First, notice that one can suppose that $G=0$ by taking $\omega^2+G\omega^1$ instead of $\omega^2$ in the
equations \eqref{equations_B64}. Now,
let $\{\omega^1,\omega^2,\omega^3\}$ be a (1,0)-basis satisfying (\ref{equations_B64}) with $G=0$,
and consider the new (1,0)-basis
$$
\{\sigma^1=\omega^1,\ \ \sigma^2=icE\omega^1+\omega^2+iE\omega^3,\ \ \sigma^3=c\omega^1+\omega^3\}.
$$
A direct calculation shows that this basis satisfies equations (\ref{equations_B64}) with $c=0$ and $E=G=0$,
that is, the complex equations can always be reduced to (\ref{reduced_equations_B64}).
In particular, all the complex structures are equivalent.
\end{proof}

In the following theorem we sum up the classification of invariant complex structures with closed (3,0)-form on
solvmanifolds obtained in this section.

\begin{theorem}\label{solvmanifolds-complex-classification}
Let $M=\Gamma\backslash G$ be a $6$-dimensional solvmanifold and denote by $\frg$ the Lie algebra of $G$.
If $J$ is an invariant complex structure on $M$ with closed $(3,0)$-form, then the pair $(\frg,J)$ is isomorphic to one
and only one of the complex structures given in Propositions~$\ref{complex-moduli-g1-g2}$,
$\ref{complex-moduli-g3}$, $\ref{complex-moduli-g4-g5-g6-g7}$, $\ref{complex-moduli-g8}$
or~$\ref{complex-moduli-g9}$.
\end{theorem}

\section{Existence of special Hermitian metrics}\label{sec-metrics}

\noindent In this section we use the classification of complex structures obtained
in the previous section to study the existence of several special Hermitian metrics.
We will center our attention on SKT, generalized Gauduchon, balanced and strongly Gauduchon metrics.

Let $(M,J,g)$ be a Hermitian manifold of real dimension $2n$ with
fundamental $2$-form $F (\cdot, \cdot) = g(J \cdot, \cdot)$. Since the metric $g$ is determined
by the form $F$, from now on we will denote a Hermitian metric also by $F$.
An \emph{SKT} (\emph{strong K\"ahler with torsion}) metric is a Hermitian metric
satisfying $\partial \db F =0$ (for more details see e.g. \cite{FPS} and the references therein).

Recently, Fu, Wang and Wu \cite{FWW} introduced and studied \emph{generalized $k$-th Gauduchon} metrics,
defined by the condition
$\partial \db F^k  \wedge F^{n -k-1} =0$, where $1 \leq k \leq n -1$.
Notice that any SKT metric is a $1$-st Gauduchon metric.

For any compact Hermitian manifold $(M,J, F)$ and for any integer $1 \leq k \leq n - 1$,
in \cite{FWW} it is proved the existence of a unique constant $\gamma_k (F)$ and a
(unique up to a constant) function $v \in {\mathcal C}^{\infty} (M)$ such that
$\frac{i}{2} \partial \db (e^v F^k) \wedge F^{n - k -1} = \gamma_k (F) e^v F^n$.
The constant $\gamma_k(F)$ is invariant under biholomorphisms and its sign is an invariant of the conformal class of $F$.
Thus, $\gamma_k (F)$ is $>0$ ($= 0$, or $< 0$) if and only if there exists $\tilde F$
in the conformal class of $F$ such that
$\frac{i}{2} \partial \db  \tilde F^k \wedge \tilde F^{n - k - 1} > 0  \, (=0, \, {\mbox {or}} \, < 0)$.

On the other hand, a Hermitian metric is called \emph{balanced}
if the fundamental form $F$ satisfies that $F^{n-1}$ is a closed form,
and it is said to be \emph{strongly Gauduchon} (sG for short)
if the $(n,n-1)$-form $\partial F^{n-1}$ is $\bar\partial$-exact.
It is obvious from the definitions that balanced implies sG.
Strongly Gauduchon metrics have been introduced recently in \cite{Pop0},
whereas balanced metrics were previously considered in \cite{Mi}.

Next we study the existence of special Hermitian metrics on 6-dimensional solvmanifolds $(M=\Gamma\backslash G,J)$
endowed with an invariant complex structure $J$ with
holomorphically trivial canonical bundle.
By Proposition~\ref{holglobalform} the latter condition is equivalent to the existence of an invariant non-zero closed (3,0)-form.
Notice that the symmetrization process
can be applied to this situation to conclude that the existence of SKT, balanced or sG metrics on $M$
reduces to the level of the Lie algebra $\frg$ of $G$ (see \cite{COUV,FG} for more details).
Thus, our strategy will consist in starting with the classification of pairs $(\frg,J)$ obtained
in Theorem~\ref{solvmanifolds-complex-classification} and then find the $J$-Hermitian structures $F$
on $\frg$ that satisfy the required conditions.

Notice that given a (1,0)-basis $\{\omega^1,\omega^2,\omega^3\}$ for the complex structure $J$, a generic
Hermitian structure $F$ on the Lie algebra $\frg$ is expressed as
\begin{equation}\label{2form}
2\,F=i\,(r^2\omega^{1\bar1}+s^2\omega^{2\bar2}+t^2\omega^{3\bar3})+u\omega^{1\bar2}-\bar
u\omega^{2\bar1}+v\omega^{2\bar3}-\bar
v\omega^{3\bar2}+z\omega^{1\bar3}-\bar
z\omega^{3\bar1},
\end{equation}
where the coefficients $r^2,\,s^2,\,t^2$ are non-zero real numbers and $u,\, v,\, z\in\C$ satisfy $r^2s^2>|u|^2$, $s^2t^2>|v|^2$, $r^2t^2>|z|^2$ and $r^2s^2t^2 + 2\Real(i\bar u\bar v z)>t^2|u|^2 + r^2|v|^2 + s^2|z|^2$.

\smallskip

Firstly we study the SKT geometry.

\begin{theorem}\label{skt-geometry}
Let $(M=\Gamma\backslash G,J)$ be a $6$-dimensional solvmanifold endowed with
an invariant complex structure $J$ with
holomorphically trivial canonical bundle,
and denote by $\frg$ the Lie algebra of~$G$.
Then, $(M,J)$ has an SKT metric if and only if $\frg\cong \frg_2^0$ or $\frg_4$.
\end{theorem}

\begin{proof}
Let $F$ be a $J$-Hermitian metric given by \eqref{2form}.
We study first the existence of SKT metrics on
$\frg_1$ and $\frg_2^{\alpha}$.
The equations \eqref{equations2-1} parametrize all the complex structures $J$ on $\frg_1$ or $\frg_2^{\alpha}$,
from which we get
\begin{equation}\label{ddbar-g1-g2}
\begin{array}{rl}
\partial\db F \!\!\!&= - 2ir^2(\Real A)^2 \omega^{13\bar{1}\bar{3}} + 2 u (\Imag A)^2 \omega^{13\bar{2}\bar{3}}
\\[7pt]
&
\phantom{=}
-2 \bar{u} (\Imag A)^2 \omega^{23\bar{1}\bar{3}} - 2 is^2(\Real A)^2 \omega^{23\bar{2}\bar{3}}.
\end{array}
\end{equation}
Thus, $\partial\db F=0$ implies $\Real A=0$, and so necessarily $A=i$. In this case $F$ is SKT
if and only if $u=0$. By Proposition~\ref{complex-moduli-g1-g2} the corresponding Lie algebra is $\frg_2^{0}$.


For the Lie algebra $\frg_3$, by Proposition~\ref{complex-moduli-g3} any $J$ on $\frg_3$ is equivalent to one
complex structure $J_x$
given by \eqref{reduced_equations_e2+e11}. The (3,3)-form $\partial\db F \wedge F$
is given by
\begin{equation}\label{ddbar-g3}
\begin{array}{rl}
\partial\db F \wedge F \!\!\!&= \frac{1+4x^2}{16x^2}\left( 4x^2s^4+t^4 \right) \omega^{123\bar{1}\bar{2}\bar{3}}.
\end{array}
\end{equation}
Since this form is never zero, there is no SKT metric on $\frg_3$.


For the Lie algebras $\frg_k$ $(4\leq k\leq 7)$, using the equations \eqref{equations2-2},
which parametrize all the complex structures $J$ on $\frg_k$,
we get
\begin{equation}\label{ddbar-g4-g5-g6-g7}
\begin{array}{rl}
\partial\db F \!\!\!&= it^2(G_{11}G_{22}-|G_{12}|^2)\omega^{12\bar{1}\bar{2}} - 2ir^2(\Real A)^2 \omega^{13\bar{1}\bar{3}}
- 2 is^2(\Real A)^2 \omega^{23\bar{2}\bar{3}}
\\[7pt]
&
\phantom{=}
+2 u (\Imag A)^2 \omega^{13\bar{2}\bar{3}}-2 \bar{u} (\Imag A)^2 \omega^{23\bar{1}\bar{3}}.
\end{array}
\end{equation}
Thus, $\partial\db F=0$ implies $\Real A=0$, and from the conditions \eqref{equations2-2-cond} we have $G_{12}=0$.
Now, $\partial\db F=0$ also implies $G_{11}G_{22}=0$, and
from Proposition~\ref{complex-moduli-g4-g5-g6-g7} it follows that only $\frg_4$ admits SKT structures:
in fact, a generic $F$ given by \eqref{2form} is SKT if and only if $u=0$.


For the study of SKT metrics on $\frg_8$, instead of using the complex structure equations
\eqref{equations2-3-1}, \eqref{equations2-3-2} and \eqref{equations2-3-3}, we use the
equations \eqref{equations_nakamura} obtained in the proof
of Proposition~\ref{complex-moduli-g8}. A direct calculation shows that
\begin{equation}\label{ddbar-g8}
\partial\db F \wedge F=2\left( r^2s^2(1+\mathfrak{Re}(A)^2) + |u|^2\mathfrak{Im}(A)^2 \right) \omega^{123\bar{1}\bar{2}\bar{3}},
\end{equation}
in particular, this form does not depend on the complex coefficients $B,C$ in~\eqref{equations_nakamura}.
The form $\partial\db F \wedge F$ never vanishes, so there is no SKT metric on $\frg_8$.


Finally, for the Lie algebra $\frg_9$,
from the complex equations \eqref{reduced_equations_B64} in Proposition~\ref{complex-moduli-g9} it follows
\begin{equation}\label{ddbar-g9}
\begin{array}{rl}
\partial\db F \wedge F \!\!\!&= \left( |v|^2+\frac{s^4}{8} \right) \omega^{123\bar{1}\bar{2}\bar{3}} \not= 0,
\end{array}
\end{equation}
so the Lie algebra $\frg_9$ does not admit SKT metrics.
\end{proof}


\begin{remark}\label{new-SKT-example}
{\rm
In the previous theorem we have proved that any complex structure with non-zero closed (3,0)-form on $\frg_2^{0}$
or $\frg_4$ admits SKT metrics.
Moreover, a generic metric $F$ given by \eqref{2form}
satisfies the SKT condition with respect to the complex equations \eqref{equations2-1-g2-0} for $(\frg_2^{0},J)$,
or \eqref{equations2-2-g4} for $(\frg_4,J_{\pm})$, if and only if $u=0$, so in both cases the SKT metrics are given by
$$
2\,F=i\,(r^2\omega^{1\bar1}+s^2\omega^{2\bar2}+t^2\omega^{3\bar3})+v\omega^{2\bar3}-\bar
v\omega^{3\bar2}+z\omega^{1\bar3}-\bar
z\omega^{3\bar1},
$$
where the coefficients $r^2,\,s^2,\,t^2$ are non-zero real numbers
and $v,\, z\in\C$ satisfy $r^2s^2t^2 > r^2|v|^2 + s^2|z|^2$.

Whereas it is known that the Lie algebra $\frg_2^{0}$ admits SKT metrics (actually it admits Calabi-Yau metrics, for instance
any $F=\frac{i}{2}(r^2\omega^{1\bar1}+s^2\omega^{2\bar2}+t^2\omega^{3\bar3})$ is K\"ahler),
however a solvmanifold based on $\frg_4$ provides, as far as we know, a new example of 6-dimensional compact SKT manifold.

We recall that  a  complex structure $J$ on  a symplectic manifold $(M,\omega)$ is said
to be tamed by the symplectic form $\omega$ if
$\omega(X,JX)>0$
for any non-zero vector field $X$ on $M$.
The pair $(\omega, J)$ is also called a \emph{Hermitian-symplectic structure} in~\cite{ST}.
By \cite[Proposition 2.1]{EFV} the existence of a Hermitian-symplectic structure
on a complex manifold  $(M, J)$ is equivalent to the existence of a $J$-compatible
SKT metric $g$ whose fundamental form $F$ satisfies
$\partial F =\overline \partial\beta$ for some $\partial$-closed $(2,0)$-form $\beta$.
As a consequence of Theorem~\ref{skt-geometry}
we have that a $6$-dimensional solvmanifold $(M=\Gamma\backslash G,J)$ with $J$ invariant and
holomorphically trivial canonical bundle, has a symplectic form $\omega$ taming $J$
if and only if $\frak g \cong \frg_2^{0}$ and $(J, \omega)$ is a K\"ahler structure.
By \cite{EFV} if $(\Gamma\backslash G, J)$ admits a non-invariant symplectic form taming $J$,
then there exists an invariant one. So we can immediately exclude
the solvmanifolds $\Gamma\backslash G$ with $\frak g\cong \frak g_4$ since $\frak g_4$
does not admit any symplectic form. For the solvmanifolds $\Gamma\backslash G$ with $\frak g \cong \frak \frg_2^{0}$
by a direct computation we have that $\partial F =\overline \partial\beta$,
for some $\partial$-closed $(2,0)$-form $\beta$, if and  only if $dF=0$.
}
\end{remark}


In the following result we study the existence of 1-st Gauduchon metrics.

\begin{theorem}\label{gamma1-geometry}
Let $(M=\Gamma\backslash G,J)$ be a $6$-dimensional solvmanifold endowed with
an invariant complex structure $J$ with
holomorphically trivial canonical bundle, and let $F$ be an invariant $J$-Hermitian metric on $M$.
If $\frg$ denotes the Lie algebra of~$G$, then we have:
\begin{enumerate}
\item[{\rm (i)}] If $\frg\cong \frg_1,\frg_2^{\alpha} (\alpha>0),\frg_3,\frg_5,\frg_7,\frg_8$ or $\frg_9$,
then $\gamma_1>0$ for any $(J,F)$.
\item[{\rm (ii)}] If $\frg\cong \frg_2^0$ or $\frg_4$,
then $\gamma_1\geq0$ for any $(J,F)$; moreover, an invariant Hermitian metric is 1-st Gauduchon if and only if it is SKT.
\item[{\rm (iii)}] If $\frg\cong \frg_6$ then there exist invariant Hermitian metrics such that $\gamma_1>0$, $=0$ or $<0$;
in particular, there are invariant 1-st Gauduchon metrics which are not SKT.
\end{enumerate}
\end{theorem}

\begin{proof}
Let $F$ be an invariant $J$-Hermitian metric given by \eqref{2form}.
Then, $F^3=-\frac34 \det(F)\, \omega^{123\bar{1}\bar{2}\bar{3}}$, where
$$
\det(F)=
\left|\!\!\! \begin{array}{ccc}
i\,r^2 & u & z \\
-\overline{u} & i\,s^2 & v \\
-\overline{z} & -\overline{v} & i\,t^2
\end{array} \!\right|.
$$
Notice that $i \det(F)>0$. Now, if
$$
\partial\db F \wedge F = \mu\, \omega^{123\bar{1}\bar{2}\bar{3}}
$$
then $\frac{i}{2}\partial\db F \wedge F=\frac{2\mu}{3i\det(F)} F^3$, which implies that
$$
\gamma_1(F)>0, =0 \mbox{ or} <0 \ \ \mbox{ if and only if }\ \  \mu>0, =0 \mbox{ or} <0.
$$
In what follows we will compute $\mu$ for any triple $(\frg,J,F)$, and study its possible signs.

For the Lie algebras $\frg_1$ and $\frg_2^{\alpha}$,
from \eqref{ddbar-g1-g2} it follows
$$
\begin{array}{rl}
\partial\db F \wedge F \!\!\!&= 2\left( r^2s^2(\Real A)^2+|u|^2(\Imag A)^2 \right) \omega^{123\bar{1}\bar{2}\bar{3}}.
\end{array}
$$
Therefore, $\gamma_1(F)\geq 0$ for any $F$. Moreover, $\gamma_1(F)\geq 0$ if and only if $\Real A=0$ and $u=0$,
which corresponds precisely to SKT metric on $\frg_2^0$.

From \eqref{ddbar-g3}, \eqref{ddbar-g8} and \eqref{ddbar-g9} it follows that $\gamma_1>0$ for any $(J,F)$ on
$\frg_3$, $\frg_8$ and $\frg_9$.

For the Lie algebras $\frg_k$ $(4\leq k\leq 7)$, using \eqref{ddbar-g4-g5-g6-g7} we get
$$
\begin{array}{rl}
2\partial\db F \wedge F \!\!\!&= \left[ 4r^2s^2(\Real A)^2+4|u|^2(\Imag A)^2-t^4(G_{11}G_{22}-|G_{12}|^2) \right]
\omega^{123\bar{1}\bar{2}\bar{3}}.
\end{array}
$$

Let us consider first $\frg_4$. By \eqref{equations2-2-g4}
we can take $A=i$, $G_{11}=\pm 1$ and $G_{12}=G_{22}=0$,
so $2\partial\db F \wedge F= 4|u|^2 \omega^{123\bar{1}\bar{2}\bar{3}}$.
This implies that $\gamma_1\geq0$, and it is equal to zero if and only if the structure is SKT.
This completes the proof of (i).

For the Lie algebra $\frg_5$, by \eqref{equations2-2-g5}
we have that $A=G_{12}=1$ and $G_{11}=G_{22}=0$,
so $2\partial\db F \wedge F= (4r^2s^2+t^4) \omega^{123\bar{1}\bar{2}\bar{3}}$
and $\gamma_1>0$.

Similarly, using \eqref{equations2-2-g7},
for $\frg_7$ we can take $A=i$, $G_{12}=0$ and
$(G_{11},G_{22})=(-1,1)$ or $(1,-1)$. Therefore, $2\partial\db F \wedge F= (t^4+4|u|^2) \omega^{123\bar{1}\bar{2}\bar{3}}$
and thus $\gamma_1>0$.
This completes the proof of (ii).

Finally, to prove {\rm (iii)}, by \eqref{equations2-2-g6}
we consider $A=i$, $G_{12}=0$ and
$G_{11}=G_{22}=1$. Since $2\partial\db F \wedge F= (4|u|^2-t^4) \omega^{123\bar{1}\bar{2}\bar{3}}$,
we conclude that on $\frg_6$ there exist Hermitian metrics such that $\gamma_1>0$, $=0$ or $<0$,
depending on the sign of $4|u|^2-t^4$.
\end{proof}


\begin{remark}\label{new-gamma1-example}
{\rm
Notice that the symmetrization process cannot be applied to the 1-st Gauduchon condition on
the solvmanifold $M=\Gamma\backslash G$ in order to reduce
the problem to the Lie algebra level, so the previous theorem studies only the existence of
\emph{invariant} 1-st Gauduchon metrics.

On the other hand, it is worthy to remark that on the solvmanifold $M=\Gamma\backslash G$ with Lie algebra $\frg\cong\frg_6$
there exist invariant 1-st Gauduchon metrics, although $M$ does not admit any SKT metric.
In fact, with respect to the complex equations \eqref{equations2-2-g6},
any invariant Hermitian metric $F$ given by \eqref{2form}
with $|u|=\frac{t^2}{2}$ is 1-st Gauduchon, however there is no SKT metric by Theorem~\ref{skt-geometry}.
This is in deep contrast with the nilpotent case, because any invariant 1-st Gauduchon metric on a 6-nilmanifold
is necessarily SKT (see \cite[Proposition 3.3]{FU}).
}
\end{remark}

In the following result we study the existence of balanced Hermitian metrics. In particular,
new examples of balanced solvmanifolds are found.

\begin{theorem}\label{balanced-geometry}
Let $(M=\Gamma\backslash G,J)$ be a $6$-dimensional solvmanifold endowed with
an invariant complex structure $J$ with
holomorphically trivial canonical bundle, and
denote by $\frg$ the Lie algebra of~$G$.
If $(M,J)$ has a balanced metric then $\frg\cong \frg_1$, $\frg_2^{\alpha}$, 
$\frg_3$, $\frg_5$, $\frg_7$ or $\frg_8$.

Moreover, in such cases, any $J$ admits balanced metrics except for the complex structures which are isomorphic to
\eqref{equations2-3-1} or \eqref{equations2-3-2} on $\frg_8$.
\end{theorem}

\begin{proof}
Since a $J$-Hermitian metric $F$ given by \eqref{2form} is balanced if and only if $\partial F^2=0$,
next we compute the (3,2)-form $\partial F^2$ for each Lie algebra $\frg$.

For the existence of balanced metrics on
$\frg_1$ and $\frg_2^{\alpha}$, from the complex structure equations
\eqref{equations2-1} it follows
\begin{equation}\label{partialF2-g1-g2}
\begin{array}{rl}
2\, \partial F^2 \!\!\!&= (ir^2z+\bar{u}v)\bar{A}\, \omega^{123\bar{1}\bar{3}}
+(is^2v -uz)\bar{A}\, \omega^{123\bar{2}\bar{3}}.
\end{array}
\end{equation}
Since $A$ is non-zero, this form vanishes if and only if $is^2v -uz=0$ and
$ir^2z+\bar{u}v=0$. Now, $r^2s^2-|u|^2>0$ implies that these conditions are equivalent to $v=z=0$.


For the Lie algebra $\frg_3$, a direct calculation using the complex equations \eqref{reduced_equations_e2+e11} shows
\begin{equation}\label{partialF2-g3}
\begin{array}{rl}
2\, \partial F^2  =\!\!\!&
-\frac{1}{2x} \left( t^2\,\Real u+\Imag(\bar{v}z)-x(it^2u+\bar{v}z) \right) \omega^{123\bar{1}\bar{2}}
\\[6pt]
\!\!\!&
+2x \left( s^2\,\Real z-\Imag(uv)+\frac{is^2z-uv}{4x} \right) \omega^{123\bar{1}\bar{3}}.
\end{array}
\end{equation}
Thus, the form $F^2$ is closed if and only if
$$
\left\{
\begin{array}{l}
it^2u+\bar{v}z=(t^2\Real u+\Imag(\bar{v}z))/x,\\[5pt]
is^2z-uv=-4x(s^2\Real z-\Imag(uv)).
\end{array}
\right.
$$
Notice that since $x$ is real, we have that both $it^2u+\bar{v}z$ and $is^2z-uv$ are also real numbers.
But this implies that $t^2\Real u+\Imag(\bar{v}z)=0$ and $s^2\Real z-\Imag(uv)=0$, and so the system
above is homogeneous. Finally, since $s^2t^2-|v|^2>0$ necessarily $u=z=0$.


For the Lie algebras $\frg_k$ $(4\leq k\leq 7)$, from equations \eqref{equations2-2} we have
\begin{equation}\label{partialF2-g4-g5-g6-g7}
\begin{array}{rl}
2\, \partial F^2 \!\!\!&= \left[(s^2t^2-|v|^2)G_{11}+(r^2t^2-|z|^2)G_{22} \right.\\[3pt]
\!\!\!&\phantom{iiia}\left.+(v\bar{z}-it^2\bar{u})G_{12}+(\bar{v}z+it^2u)\overline{G}_{12}\, \right] \omega^{123\bar{1}\bar{2}}\\[5pt]
\!\!\!&\phantom{ii}+(ir^2v+\bar{u}z)\overline{A}\, \omega^{123\bar{1}\bar{3}}
+(is^2z-uv)\overline{A}\, \omega^{123\bar{2}\bar{3}}.
\end{array}
\end{equation}
Since $A$ is non-zero and $r^2s^2-|u|^2>0$, the coefficients of $\omega^{123\bar{1}\bar{3}}$
and $\omega^{123\bar{2}\bar{3}}$ vanish if and only if $v=z=0$. The latter conditions
reduce the expression of the form to
$$
2\, \partial F^2 = t^2 \left( s^2\,G_{11}+r^2\,G_{22}-i\bar{u}\,G_{12}+iu\,\overline{G}_{12} \right)\, \omega^{123\bar{1}\bar{2}}.
$$
Now, we can use the complex classification given in Proposition~\ref{complex-moduli-g4-g5-g6-g7}
to conclude that the only possibilities to get a closed form $F^2$ are, either $G_{12}=0$ and
$(G_{11},G_{22})=(1,-1),(-1,1)$,
or $G_{11}=G_{22}=0$ and $G_{12}=1$. The first case corresponds to $\frg_7$ and the
coefficients $r^2$ and $s^2$ in the metric must be equal,
whereas the second case corresponds to $\frg_5$ with metric coefficient $u\in \mathbb{R}$.


For the study of balanced Hermitian metrics on $\frg_8$, by the complex
equations \eqref{equations_nakamura}, a direct calculation shows that
\begin{equation}\label{partialF2-g8}
\begin{array}{rl}
2\,\partial F^2= & - \left[ (ir^2v+\bar{u}z)(\overline{A}-i) + (r^2s^2-|u|^2) \overline{C} \right] \omega^{123\bar{1}\bar{3}}\\[7pt]
& + \left[ (uv-is^2z)(\overline{A}-i) + (r^2s^2-|u|^2) \overline{B} \right] \omega^{123\bar{2}\bar{3}}.
\end{array}
\end{equation}
Since $r^2s^2-|u|^2\not=0$, the structure $(J,F)$ is balanced if and only if
$$
\begin{array}{l}
B=-\frac{is^2\bar{z}+\bar{u}\bar{v}}{r^2s^2-|u|^2} (A+i),\quad\quad
C=\frac{ir^2\bar{v}-u\bar{z}}{r^2s^2-|u|^2} (A+i).
\end{array}
$$
It follows from Proposition~\ref{complex-moduli-g8} that the complex structures \eqref{equations2-3-1} and \eqref{equations2-3-2}
do not admit balanced metrics, because $A=-i$ but $B$ is not zero.
However, any complex structure in the family \eqref{equations2-3-3} has balanced Hermitian metrics because $B=C=0$.
In fact, if $A\neq-i$ then
the metric \eqref{2form} is balanced if and only if $v=z=0$, and for $A=-i$ (i.e. the complex structure is bi-invariant)
any metric is balanced.


In the case of the Lie algebra $\frg_9$,
from the complex equations \eqref{reduced_equations_B64} it follows
\begin{equation}\label{partialF2-g9}
\begin{array}{rl}
4\, \partial F^2 =\!\!\!& \left( i\,\bar{u}\bar{v}-s^2\bar{z} \right) \omega^{123\bar{1}\bar{2}}
- \left( i\,v\bar{z} +t^2\bar{u}-uv+i\,s^2z \right) \omega^{123\bar{1}\bar{3}}\\[6pt]
\!\!\!&+2(|u|^2-r^2s^2)\omega^{123\bar{2}\bar{3}},
\end{array}
\end{equation}
which implies that the component of $\partial F^2$ in $\omega^{123\bar{2}\bar{3}}$ is nonzero,
so there are not balanced Hermitian metrics.

Finally, notice that for the Lie algebras $\frg_1$, $\frg_2^{\alpha}$,
$\frg_3$, $\frg_5$ and $\frg_7$ we have proved above that any complex structure $J$ admits balanced Hermitian metrics.
However, for the Lie algebras $\frg_8$, a complex structure $J$ admits balanced metric if and only if
it is isomorphic to one in the family \eqref{equations2-3-3}.
\end{proof}


Next we prove that the non-sufficient necessary condition for the existence of balanced metrics found in
Theorem~\ref{balanced-geometry} is necessary and sufficient for the existence of sG metrics.

\begin{theorem}\label{sG-geometry}
Let $(M=\Gamma\backslash G,J)$ be a $6$-dimensional solvmanifold endowed with
an invariant complex structure $J$ with
holomorphically trivial canonical bundle, and denote by $\frg$ the Lie algebra of~$G$.
Then, $(M,J)$ has an sG metric if and only if $\frg\cong \frg_1$, $\frg_2^{\alpha}$, 
$\frg_3$, $\frg_5$, $\frg_7$ or $\frg_8$.

Moreover, if $\frg\cong \frg_1$, $\frg_2^{\alpha}$, 
$\frg_3$ or $\frg_8$, then any invariant Hermitian metric is sG.
\end{theorem}

\begin{proof}
Since balanced implies sG, by Theorem~\ref{balanced-geometry} we know
that if $\frg\cong \frg_1$, $\frg_2^{\alpha}$,
$\frg_3$, $\frg_5$, $\frg_7$ or $\frg_8$, then there exist sG metrics.
Moreover, any $J$ on the Lie algebras $\frg_1$, $\frg_2^{\alpha}$,
$\frg_3$, $\frg_5$ and $\frg_7$ admits sG metrics. We prove next that
there are not sG metrics on $\frg_4$, $\frg_6$ and $\frg_9$.

From \eqref{equations2-2} we have
$\db(\Lambda^{3,1})=\langle\omega^{123\bar{1}\bar{3}},\omega^{123\bar{2}\bar{3}}\rangle$, and
by~\eqref{partialF2-g4-g5-g6-g7} the (3,2)-form $\partial F^2$
is a combination of $\omega^{123\bar{1}\bar{2}}$, $\omega^{123\bar{1}\bar{3}}$ and $\omega^{123\bar{2}\bar{3}}$.
Hence, the existence of sG metric is equivalent to the vanishing of the coefficient of $\omega^{123\bar{1}\bar{2}}$ in $\partial F^2$.
By \eqref{equations2-2-g4}, the Lie algebra $\frg_4$ corresponds to
$A=i$, $G_{11}=\pm1$ and $G_{12}=G_{22}=0$, so the coefficient of $\omega^{123\bar{1}\bar{2}}$
is equal to $\pm(s^2t^2-|v|^2)$, which is never zero.
On the other hand, by \eqref{equations2-2-g6} the Lie algebra $\frg_6$ corresponds to
$A=i$, $G_{11}=G_{22}=1$ and $G_{12}=0$,
and the coefficient of $\omega^{123\bar{1}\bar{2}}$ is $(s^2t^2-|v|^2)+(r^2t^2-|z|^2)$,
which is strictly positive.
In conclusion, there do not exist sG metrics for $\frg_4$ or $\frg_6$.

For the Lie algebra $\frg_9$, equations \eqref{reduced_equations_B64} imply
$$
\db\omega^{123\bar{1}}=0,\qquad
\db\omega^{123\bar{2}}=(i/2)\, \omega^{123\bar{1}\bar{2}},\qquad
\db\omega^{123\bar{3}}=-(i/2)\, \omega^{123\bar{1}\bar{3}},
$$
therefore $\db \Lambda^{3,1}=\langle \omega^{123\bar{1}\bar{2}},\omega^{123\bar{1}\bar{3}} \rangle$.
By \eqref{partialF2-g9} we have that
the component of $\partial F^2$ in $\omega^{123\bar{2}\bar{3}}$ is nonzero,
so $\partial F^2\not\in \db \Lambda^{3,1}$ and $F$ is never sG.
Thus, there do not exist sG metrics for $\frg_9$.

To finish the proof it remains to see that any pair $(J,F)$ on
$\frg_1$, $\frg_2^{\alpha}$, $\frg_3$ and $\frg_8$ is sG.
By Proposition~\ref{complex-moduli-g1-g2} and \eqref{partialF2-g1-g2}
a direct calculation implies $\partial F^2 \in \db(\Lambda^{3,1})$,
so any $(J,F)$ on $\frg_1$ or $\frg_2^{\alpha}$ is sG. For $\frg_3$ (resp. $\frg_8$)
we also have $\partial F^2 \in \db(\Lambda^{3,1})$ for any Hermitian structure $(J,F)$
by Proposition~\ref{complex-moduli-g3} and \eqref{partialF2-g3}
(resp. Proposition~\ref{complex-moduli-g8} and~\eqref{partialF2-g8}).
\end{proof}

\section{Holomorphic deformations}\label{hol-deform}

\noindent In this section we study some properties related to the existence of balanced metrics
under deformation of the complex structure. In what follows, $(M,J_a)_{a\in \Delta}$,
$\Delta$ being an open disc around the origin in $\mathbb{C}$, will denote a holomorphic family of compact complex manifolds.
We briefly recall that a property is said to be \emph{open} under holomorphic deformations
if when it holds for a given compact complex manifold $(M,J_{0})$,
then $(M,J_a)$ also has that property for all $a\in\Delta$ sufficiently close to $0$.
On the other hand, a property is said to be \emph{closed} under holomorphic deformations
if whenever $(M,J_{a})$ has that property for all $a\in \Delta\setminus \{0\}$
then the property also holds for the central limit $(M,J_{0})$.

Concerning the property of existence of balanced Hermitian metrics, Alessandrini and Bassanelli proved in \cite{AB} (see also \cite{FG})
that it is not deformation open.
In contrast to the balanced case, the sG property is open under holomorphic deformations \cite{Pop2}.
However, in \cite{COUV} it is shown that the sG property and the balanced property
of compact complex manifolds are not closed under holomorphic deformations.
More concretely, there exists a holomorphic family of compact complex manifolds $(M,J_a)_{a\in \Delta}$ such that
$(M,J_a)$ has balanced metric for any $a\not=0$ but the central limit $(M,J_0)$ does not admit any sG metric, which
provides a counterexample to the Popovici and Demailly closedness conjectures formulated in \cite{Pop2}.

On the other hand, recall that a compact complex manifold $M$ is said to satisfy the $\partial\db$-lemma if for any
$d$-closed form $\alpha$ of pure type on $M$, the following exactness properties are equivalent:

\vskip.2cm

\hskip1cm $\alpha$ is $d$-exact $\Longleftrightarrow$ $\alpha$ is $\partial$-exact
$\Longleftrightarrow$ $\alpha$ is $\db$-exact
$\Longleftrightarrow$ $\alpha$ is $\partial\db$-exact.

\vskip.2cm

\noindent
Under this strong condition, the existence of sG metric in the central limit
is guaranteed:

\begin{proposition}\cite[Proposition 4.1]{Pop09}\label{Popovici}
If the $\partial \db$-lemma holds on $(M,J_a)$ for every $a\in \Delta\setminus \{0\}$,
then $(M,J_0)$ has an sG metric.
\end{proposition}

An interesting problem is if the conclusion in the above proposition holds under weaker conditions than the $\partial \db$-lemma.
In \cite[Corollary 4.5]{LUV} it is proved that the vanishing of some complex invariants, which are closely related to the
$\partial \db$-lemma, is not sufficient to ensure the existence of an sG metric in the central limit.

Another problem related to Proposition~\ref{Popovici} is if the central limit admits a Hermitian metric, stronger than sG,
under the $\partial \db$-lemma condition. Our aim in this section is to construct a holomorphic family of compact complex manifolds
$(M,J_a)_{a\in \Delta}$ such that $(M,J_a)$ satisfies the $\partial\db$-lemma and admits balanced metric
for any $a\not=0$, but the central limit neither satisfies the $\partial\db$-lemma nor admits balanced metric.
As far as we know, this is the first known complex deformation with this behaviour.
The construction is based on the balanced Hermitian geometry of $\frg_8$ studied in Theorem~\ref{balanced-geometry},
which is the Lie algebra underlying Nakamura manifold.

Concerning the $\partial\db$-lemma property, it is known that it is an open property (see for instance \cite{AT} for
a proof of this fact), and recently Angella and Kasuya have proved in \cite{AK}
that the $\partial\db$-lemma is not a closed property under holomorphic deformations.
The construction in \cite{AK} consists in a suitable deformation $(M,I_t)$ of
the holomorphically parallelizable Nakamura manifold $(M,I_0)$ (notice that $(M,I_0)$ has balanced metrics).
We will use their result on the $\partial\db$-lemma for $(M,I_t)$, $t\not=0$, as a key ingredient
in the proof of the following result.

\begin{theorem}\label{counterexample-Popovici}
There exists a solvmanifold $M$ with a holomorphic family of complex structures $J_a$,
$a\in\Delta=\{a\in \mathbb{C} \mid |a|<1\}$, such that $(M,J_a)$ satisfies the $\partial\db$-lemma and admits balanced metric
for any $a\not=0$, but the central limit $(M,J_0)$ neither satisfies the $\partial\db$-lemma nor admits balanced metrics.
\end{theorem}

\begin{proof}
Let $J$ be the complex structure
on the Lie algebra $\frg_8$ defined by \eqref{equations2-3-1} in Proposition~$\ref{complex-moduli-g8}$.
By Theorem~\ref{balanced-geometry}, any complex solvmanifold 
$(M,J)$ 
with underlying Lie algebra $\frg_8$ and complex structure $J$ 
does not admit balanced metrics.
Next we will deform $J$ to an analytic family $J_a$ defined for any $a$ in the open unit disc $\Delta$ in 
$\mathbb{C}$ centered at $0$, so that $J$ is the central limit of the holomorphic deformation, i.e. $J_0=J$.

For each $a\in \mathbb{C}$ such that $|a|<1$, we consider the complex structure $J_a$ on $M$
defined by the (1,0)-basis
$$
\Phi^1=\omega^1,\ \ \Phi^2=\omega^2,\ \ \Phi^3=\omega^3+ a\, \omega^{\bar{3}}.
$$
It is easy to check that the complex structure equations are
\begin{equation}\label{eq-deform}
\left\{
\begin{array}{l}
d\Phi^1=\frac{2i}{1-|a|^2}\Phi^{13}-\frac{2ia}{1-|a|^2}\Phi^{1\bar{3}}+\frac{1}{1-|a|^2}\Phi^{3\bar{3}},\\[3pt]
d\Phi^2=-\frac{2i}{1-|a|^2}\Phi^{23}+\frac{2ia}{1-|a|^2}\Phi^{2\bar{3}},\\[3pt]
d\Phi^3=0.
\end{array}
\right.
\end{equation}
Using these equations, the (2,3)-form $\db F^2$ for a generic metric~\eqref{2form} with respect to the basis
$\{\Phi^{1},\Phi^{2},\Phi^{3}\}$
reads as
$$
2\, \db F^2= \left[ \frac{2ia(ir^2\bar{v}-u\bar{z})}{1-|a|^2} \right] \Phi^{13\bar{1}\bar{2}\bar{3}}
+\left[ \frac{2ia(is^2\bar{z}+\bar{u}\bar{v})}{1-|a|^2} + (r^2s^2-|u|^2) \right] \Phi^{23\bar{1}\bar{2}\bar{3}}.
$$
Suppose that $a\not=0$ with $|a|<1$.
If $u=v=0$ then the balanced condition reduces to solve
$$
\frac{2a\bar{z}}{1-|a|^2} =r^2, \quad \mbox{ with } r^2t^2 > |z|^2.
$$
Thus, taking $z=(1-|a|^2)r^2/(2\bar{a})$, the condition
$r^2t^2 > |z|^2$ is satisfied for any $t$ such that $t^2 > \frac{(1-|a|^2)^2r^2}{4|a|^2}$.

Therefore, we have proved that for any $J_a$, $a\in \Delta-\{0\}$, the structures
\begin{equation}\label{balanced-in-deform}
2\,F=i\,(r^2\Phi^{1\bar1}+s^2\Phi^{2\bar2}+t^2\Phi^{3\bar3})
+\frac{(1-|a|^2)r^2}{2\bar{a}}\Phi^{1\bar3}-\frac{(1-|a|^2)r^2}{2a}\Phi^{3\bar1},
\end{equation}
with $r,s\not=0$ and $t^2 > \frac{(1-|a|^2)^2r^2}{4|a|^2}$,
are balanced.

Notice that the previous argument is valid for the quotient $M$ of any lattice in the
simply-connected Lie group $G$ associated to $\frg=\frg_8$. However, to ensure the $\partial\db$-lemma
for the complex structures $J_a$ with $a\not=0$ we need to consider the lattice $\Gamma$ considered in \cite{AK}.
In fact, in \cite{AK} the authors consider
the holomorphically parallelizable Nakamura manifold $X=(\Gamma\backslash G,I_0)$, whose complex structure $I_0$ corresponds to
the complex structure $J_{-i}$ in our family \eqref{equations2-3-3} in Proposition~$\ref{complex-moduli-g8}$,
and they consider a small deformation $I_t$ given by
$$
t \frac{\partial}{\partial z_3} \otimes d \bar{z}_3 \in H^{0,1}(X;T^{1,0}X),
$$
where $z_3$ is a complex coordinate such that $\omega^3=d z_3$. By \cite[Proposition 4.1]{AK}
(see also Tables 7 and 8 in \cite{AK})
one has that
$X_t=(\Gamma\backslash G_8,I_t)$ satisfies the $\partial\db$-lemma for any $t\not=0$.
Since $I_0=J_{-i}$, in terms of complex structure
equations \eqref{equations2-3-3} for $A=-i$ the deformation $I_t$ is defined by the (1,0)-basis
$$
\Upsilon^1=\omega^1,\ \ \Upsilon^2=\omega^2,\ \ \Upsilon^3=\omega^3- t\, \omega^{\bar{3}},
$$
and the structure equations for $I_t$ are
\begin{equation}\label{eq-deform-bis}
\left\{
\begin{array}{l}
d\Upsilon^1=\frac{2i}{1-|t|^2}\Upsilon^{13}+\frac{2it}{1-|t|^2}\Upsilon^{1\bar{3}},\\[3pt]
d\Upsilon^2=-\frac{2i}{1-|t|^2}\Upsilon^{23}-\frac{2it}{1-|t|^2}\Upsilon^{2\bar{3}},\\[3pt]
d\Upsilon^3=0.
\end{array}
\right.
\end{equation}

On the other hand, it is easy to see that for any $a\not=0$ the equations \eqref{eq-deform} express with respect to the
(1,0)-basis $\{\Theta^1=\Phi^1+\frac{i}{2a}\Phi^3,\Theta^2=\Phi^2,\Theta^3=\Phi^3\}$ as
\begin{equation}\label{eq-deform-bis2}
\left\{
\begin{array}{l}
d\Theta^1=\frac{2i}{1-|a|^2}\Theta^{13}-\frac{2ia}{1-|a|^2}\Theta^{1\bar{3}},\\[3pt]
d\Theta^2=-\frac{2i}{1-|a|^2}\Theta^{23}+\frac{2ia}{1-|a|^2}\Theta^{2\bar{3}},\\[3pt]
d\Theta^3=0.
\end{array}
\right.
\end{equation}
Now, from \eqref{eq-deform-bis} and \eqref{eq-deform-bis2} we conclude that for $a\not=0$
the complex structure $J_a$ is precisely the complex structure $I_{t}$ with $t=-a$.
Therefore, for any $a\not=0$ the compact complex manifold $(M,J_a)=(\Gamma\backslash G,J_a)$
satisfies the $\partial\db$-lemma because $X_{t=-a}$ does by \cite[Proposition 4.1]{AK}.

To finish the proof, it remains to see that the central limit $J_0$ does not satisfy the
$\partial\db$-lemma. By the symmetrization process, it suffices to prove that
it is not satisfied at the Lie algebra level $(\frg_8,J_0)$.
But this is clear from the equations \eqref{equations2-3-1}, because the
form $\omega^{23}$ is $\partial$-closed, $\db$-closed and $d$-exact,
however it is not $\partial\db$-exact.
\end{proof}

\begin{remark}\label{correspondence}
{\rm
For any $a\in \Delta-\{0\}$, the equations \eqref{eq-deform-bis2} imply that the (3,0)-form $\Theta^{123}$ is closed, so the 
complex structure $J_a$ must be equivalent to a complex structure $J_A$ in~\eqref{equations2-3-3}. 
In fact, taking
$\omega^1=\Theta^1$, $\omega^2=\Theta^2$ and $\omega^3=\frac{1}{1-a}\Theta^3$, the 
complex equations~\eqref{eq-deform-bis2} express in terms of the (1,0)-basis $\{ \omega^1,\omega^2,\omega^3 \}$
by 
$$
d\omega^1=-(A-i)\omega^{13}-(A+i)\omega^{1\bar{3}},\quad
d\omega^2=(A-i)\omega^{23}+(A+i)\omega^{2\bar{3}},\quad
d\omega^3=0,
$$
with $A=\frac{i}{|a|^2-1}(1+|a|^2-2a)$.
Notice that this correspondence, together with Theorem~\ref{balanced-geometry}, assures the existence of balanced metric
for all $a\in \Delta-\{0\}$, however in the proof of Theorem~\ref{counterexample-Popovici}
we have provided in \eqref{balanced-in-deform} an explicit family of balanced metrics.
Of course, the central limit of any metric \eqref{balanced-in-deform} does not exist, since $J_0$ does not admit any
balanced metric by Theorem~\ref{balanced-geometry}.
}
\end{remark}

\begin{remark}\label{remark-deform}
{\rm
%
We can construct another deformation with central limit the complex structure $J'$ given by \eqref{equations2-3-2},
and it turns out that the deformation has the same behaviour
as for the deformation of the complex structure $J$ given by \eqref{equations2-3-1}
constructed in Theorem~\ref{counterexample-Popovici}.
Therefore, the complex structures $J$ and $J'$ given by \eqref{equations2-3-1} and \eqref{equations2-3-2}, respectively,
are the central limits of complex structures that satisfy the $\partial\db$-lemma.
Notice that this is consistent with Proposition~\ref{Popovici} because by our Theorem~\ref{sG-geometry} both complex structures
admit sG metric.
}
\end{remark}

\section{Appendix}\label{sec-appendix}

\noindent In this appendix we include the classification of (non nilpotent) solvable Lie algebras which are unimodular
and have $b_3$ at least 2. Table 1 contains the decomposable case, whereas Table 2 refers to the
indecomposable case.

For Table 1 we use mainly the list of low dimensional Lie algebras of \cite{F-Schu1}.
The $3\oplus3$ case is the product of two 3-dimensional unimodular
(non nilpotent) solvable Lie algebras (notice that in this case $b_3$ is always $\geq 2$ by \eqref{b3}).
The $4\oplus2$ case is the product by $\mathbb{R}^2$ of a 4-dimensional unimodular
(non nilpotent) solvable Lie algebra $\frh$ satisfying $b_1(\frh)+2b_2(\frh)+b_3(\frh)\geq 2$.
Finally, the $5\oplus1$ case is the product by $\mathbb{R}$ of a 5-dimensional unimodular
(non nilpotent) solvable Lie algebra $\frh$ with $b_2(\frh)+b_3(\frh)\geq 2$.

In Table 2, the first two Lie algebras labeled as
$N_{6,18}^{0,-1,-1}$ and $N_{6,20}^{-1,-1}$
comes from the classification in \cite{T}, and they are the only unimodular solvable Lie algebras
with $b_3\geq 2$ and nilradical of dimension 4. The other Lie algebras in Table 2 are taken from
\cite{F-Schu2}.
In Table 2 we also include the column ``$\lambda(\rho)\geq 0$'' in which the symbol $\checkmark$
means that any closed 3-form $\rho$ on the Lie algebra satisfies $\lambda(\rho)\geq 0$, in particular,
$\rho$ does not give rise to an almost complex structure
(a similar study was done in \cite{F-Schu1} for any decomposable Lie algebra).

\bigskip
\medskip

\vfill\eject

{\small
\centerline{
\begin{tabular}{|l|l|c|}
\hline
\bf{Lie algebra} $\frg$& \bf{structure equations}& \bf{closed (3,0)-form} $\Psi$ \\
\hline
$\fre(2)\oplus\fre(2)$&$(0,-e^{13},e^{12},0,-e^{46},e^{45})$&$-$\\\hline
$\fre(2)\oplus\fre(1,1)$&$(0,-e^{13},e^{12},0,-e^{46},-e^{45})$&
$(e^1\!\!-\!ie^4)\!\wedge\!(e^2\!-\!2i(e^2\!\!-\!e^6))\!\wedge\!(e^3\!\!+\!i(\frac{e^3}{2}\!+\!e^5))$\\\hline
$\fre(2)\oplus\frh_3$&$(0,-e^{13},e^{12},0,0,e^{45})$&$-$\\\hline
$\fre(2)\oplus\R^3$&$(0,-e^{13},e^{12},0,0,0)$&$-$\\\hline
$\fre(1,1)\oplus\fre(1,1)$&$(0,-e^{13},-e^{12},0,-e^{46},-e^{45})$&$-$\\\hline
$\fre(1,1)\oplus\frh_3$&$(0,-e^{13},-e^{12},0,0,e^{45})$&$-$\\\hline
$\fre(1,1)\oplus\R^3$&$(0,-e^{13},-e^{12},0,0,0)$&$-$\\\hline
$A_{4,2}^{-2}\oplus\R^2$&$(-2e^{14},e^{24}+e^{34},e^{34},0,0,0)$&$-$\\\hline
$A_{4,5}^{\alpha,-1-\alpha}\oplus\R^2$&$(e^{14},\alpha e^{24},-(1+\alpha)e^{34},0,0,0)$&$-$\\
$-1<\alpha\leq-\frac{1}{2}$&&\\\hline
$A_{4,6}^{\alpha,-\frac{\alpha}{2}}\oplus\R^2$&$(\alpha e^{14},-\frac{\alpha}{2} e^{24}\!+\!e^{34},-e^{24}\!-\!\frac{\alpha}{2} e^{34},0,0,0)$&$-$\\
$\alpha>0$& &\\\hline
$A_{4,8}\oplus\R^2$&$(e^{23},e^{24},-e^{34},0,0,0)$&$-$\\\hline
$A_{4,10}\oplus\R^2$&$(e^{23},e^{34},-e^{24},0,0,0)$&$-$\\\hline
$A_{5,7}^{-1,-1,1}\oplus\R$&$(e^{15},-e^{25},-e^{35},e^{45},0,0)$&$(e^1-ie^4)\wedge(e^2-ie^3)\wedge(e^5-ie^6)$\\\hline
$A_{5,7}^{-1,\beta,-\beta}\oplus\R$&$(e^{15},-e^{25},\beta e^{35},-\beta e^{45},0,0)$&$-$\\
$0<\beta<1$&&\\\hline
$A_{5,8}^{-1}\oplus\R$&$(e^{25},0,e^{35},-e^{45},0,0)$&$-$\\\hline
$A_{5,9}^{-1,-1}\oplus\R$&$(e^{15}+e^{25},e^{25},-e^{35},-e^{45},0,0)$&$-$\\\hline
$A_{5,13}^{-1,0,\gamma}\oplus\R$&$(e^{15},-e^{25},\gamma e^{45},-\gamma e^{35},0,0)$&$-$\\
$\gamma>0$&&\\\hline
$A_{5,14}^{0}\oplus\R$&$(e^{25},0,e^{45},-e^{35},0,0)$&$-$\\\hline
$A_{5,15}^{-1}\oplus\R$&$(e^{15}+e^{25},e^{25},-e^{35}+e^{45},-e^{45},0,0)$&$-$\\\hline
$A_{5,17}^{0,0,\gamma}\oplus\R$&$(e^{25},-e^{15},\gamma e^{45},-\gamma e^{35},0,0)$&$-$\\
$0<\gamma<1$&&\\\hline
$A_{5,17}^{\alpha,-\alpha,1}\oplus\R$&$(\alpha e^{15}+e^{25},-e^{15}+\alpha e^{25},$&$(e^1-ie^2)\wedge(e^4-ie^3)\wedge(e^6-ie^5)$\\
$\alpha\geq 1$& \ $-\alpha e^{35}+e^{45},-e^{35}-\alpha e^{45},0,0)$ &\\\hline
$A_{5,18}^{0}\oplus\R$&$(e^{25}+e^{35},-e^{15}+e^{45},e^{45},-e^{35},0,0)$&$-$\\\hline
$A_{5,19}^{-1,2}\oplus\R$&$(-e^{15}+e^{23},e^{25},-2 e^{35},2 e^{45},0,0)$&$-$\\\hline
$A_{5,19}^{1,-2}\oplus\R$&$(e^{15}+e^{23},e^{25},0,-2 e^{45},0,0)$&$-$\\\hline
$A_{5,20}^{0}\oplus\R$&$(e^{23}+e^{45},e^{25},-e^{35},0,0,0)$&$-$\\\hline
$A_{5,26}^{0,\pm 1}\oplus\R$&$(e^{23}\pm e^{45},-e^{35},e^{25},0,0,0)$&$-$\\\hline
$A_{5,33}^{-1,-1}\oplus\R$&$(e^{14},e^{25},-e^{34}-e^{35},0,0,0)$&$-$\\\hline
$A_{5,35}^{0,-2}\oplus\R$&$(-2e^{14},e^{24}+e^{35},-e^{25}+e^{34},0,0,0)$&$-$\\\hline
\end{tabular}
}
}

\vspace{5pt}
\centerline{{\bf Table 1.} Decomposable unimodular (non-nilpotent) solvable}

\centerline{Lie algebras with $b_3\geq 2$.}


 {\footnotesize
\centerline{
\begin{tabular}{|l|l|c|c|}
\hline
\bf{Lie algebra} $\frg$& \bf{structure equations}& \bf{closed (3,0)-form} $\Psi$ &$\lambda(\rho)\geq 0$\\
\hline
$N_{6,18}^{0,-1,-1}$&$(e^{16}-e^{25},e^{15}+e^{26},-e^{36}+e^{45},$&$(e^1+ie^2)\wedge(e^3+ie^4)\wedge(e^5-ie^6)$&$-$\\
&$-e^{35}-e^{46},0,0)$&&\\\hline
$N_{6,20}^{-1,-1}$&$(-e^{56},-e^{25}-e^{26},-e^{36},e^{45},0,0)$&$-$&$-$\\\hline
$A_{6,13}^{-1,b,-2b+1}$&$((b-1)e^{16}+e^{23},-e^{26},be^{36},e^{46},$&$-$&$-$\\
$b\notin\{-1,0,\frac{1}{2},1,2\}$&$(1-2b)e^{56},0)$&&\\\hline
$A_{6,13}^{a,-2a,2a-1}$&$(-ae^{16}+e^{23},a e^{26},-2ae^{36},e^{46},$&$-$&$\checkmark$\\
$a\notin\{-1,0,\frac{1}{3},\frac{1}{2}\}$&$(2a-1)e^{56},0)$&&\\\hline
$A_{6,13}^{a,-a,-1}$&$(e^{23},ae^{26},-ae^{36},e^{46},-e^{56},0)$&$-$&$\checkmark$\\
$a>0,\ a\neq1$&&&\\\hline
$A_{6,13}^{a,b,c},\ (a,b,c)\in\{(0,-1,1),$&$((a+b)e^{16}+e^{23},ae^{26},be^{36},e^{46},$&$-$&$-$\\
$(-1,1,-1),(-1,-1,3),$&$ce^{56},0)$&&\\
$(-1,2,-3)\}$&&&\\\hline
$A_{6,14}^{\frac{1}{3},-\frac{2}{3}}$&$(-\frac{1}{3}e^{16}+e^{23}+e^{56},\frac{1}{3}e^{26},-\frac{2}{3}e^{36},$&$-$&$\checkmark$\\
&$e^{46},-\frac{1}{3}e^{56},0)$&&\\\hline
$A_{6,14}^{-1,\frac{2}{3}}$&$(-\frac{1}{3}e^{16}+e^{23}+e^{56},-e^{26},\frac{2}{3}e^{36},$&$-$&$-$\\
&$e^{46},-\frac{1}{3}e^{56},0)$&&\\\hline
$A_{6,15}^{-1}$&$(e^{23},e^{26},-e^{36},e^{26}+e^{46},e^{36}-e^{56},0)$&$-$&$-$\\\hline
$A_{6,17}^{0,-\frac{1}{2}}$&$(-\frac{1}{2}e^{16}+e^{23},-\frac{1}{2}e^{26},0,e^{36},e^{56},0)$&$-$&$-$\\\hline
$A_{6,18}^{a,b}$&$((a+1)e^{16}+e^{23},ae^{26},e^{36},e^{36}+e^{46},$&$-$&$\checkmark$\\
$(a,b)\in\{(-\frac{1}{2},-2),(-2,1)\}$&$be^{56},0)$&&\\\hline
$A_{6,18}^{a,b}$&$((a+1)e^{16}+e^{23},ae^{26},e^{36},e^{36}+e^{46},$&$-$&$-$\\
$(a,b)\in\{(-1,-1),(-3,3)\}$&$be^{56},0)$&&\\\hline
$A_{6,21}^{a,b},\,(a,b)\in\{(0,-1), $&$(2ae^{16}+e^{23},ae^{26},e^{26}+ae^{36},e^{46},$&$-$&$-$\\
$(-1,3),(-\frac{1}{3},\frac{1}{3})\}$&$be^{56},0)$&&\\\hline
$A_{6,25}^{a,b}$&$((b+1)e^{16}+e^{23},e^{26},be^{36},ae^{46},$&$-$&$\checkmark$\\
$(a,b)\in\{(0,-1),(-\frac{1}{2},-\frac{1}{2})\}$&$e^{46}+ae^{56},0)$&&\\\hline
$A_{6,25}^{-1,0}$&$(e^{16}+e^{23},e^{26},0,-e^{46},e^{46}-e^{56},0)$&$-$&$-$\\\hline
$A_{6,26}^{-1}$&$(e^{23}+e^{56},e^{26},-e^{36},0,e^{46},0)$&$-$&$-$\\\hline
$A_{6,32}^{0,b,-b}$&$(e^{23},-e^{36},e^{26},be^{46},-be^{56},0)$&$-$&$\checkmark$\\
$b>0$&&&\\\hline
$A_{6,34}^{0,0,\epsilon}$&$(e^{23}+\epsilon e^{56},-e^{36},e^{26},0,e^{46},0)$&$-$&$\checkmark$\\
$\epsilon\in\{0,1\}$&&&\\\hline
$A_{6,35}^{a,b,c},\, a>0,$& $((a+b)e^{16}+e^{23},ae^{26},be^{36},$ &$-$&$\checkmark$\\
$(b,c)\in\{(-2a,a),(-a,0)\}$ & \ $ce^{46}-e^{56},e^{46}+ce^{56},0)$ &&\\\hline
$A_{6,36}^{0,0}$&$(e^{23},0,e^{26},-e^{56},e^{46},0)$&$-$&$-$\\\hline
$A_{6,37}^{0,0,c}$&$(e^{23},-e^{36},e^{26},-ce^{56},ce^{46},0)$&$-$&$\checkmark$\\
$c>0,\ c\neq1$&&&\\\hline
$A_{6,37}^{0,0,1}$&$(e^{23},-e^{36},e^{26},-e^{56},e^{46},0)$&$(e^1+ie^6)\wedge(e^2-ie^3)\wedge(e^4+ie^5)$&$-$\\\hline
$A_{6,39}^{a,b}\,,(a,b)\in\{(-1,-1),$& $((b+1)e^{16}+e^{45},e^{15}+(b+2)e^{26},$ &$-$&$-$\\
$(-\frac{5}{2},-\frac{1}{2}),(5,-3),(2,-2)\} $&$ae^{36},be^{46},e^{56},0)$ &&\\\hline
\end{tabular}
}
}


\vskip.2cm

\centerline{{\bf Table 2.} Indecomposable unimodular (non-nilpotent) solvable}

\centerline{Lie algebras with $b_3\geq 2$.}

 {\footnotesize
\centerline{
\begin{tabular}{|l|l|c|c|}
\hline
\bf{Lie algebra} $\frg$& \bf{structure equations}& \bf{closed (3,0)-form} $\Psi$ &$\lambda(\rho)\geq 0$\\
\hline
$A_{6,41}^{-1}$&$(e^{45},e^{15}+e^{26},-e^{36}+e^{46},-e^{46},e^{56},0)$&$-$&$-$\\\hline
$A_{6,54}^{a,b}$ & $(e^{16}+e^{35},be^{26}+e^{45},(1-a)e^{36},$ &$-$&$-$\\
$(a,b)\in\{(0,-1),$& \ $(b-a)e^{46},ae^{56},0)$ &&\\
$(-1,-\frac{3}{2}),(2,0)\}$&  &&\\\hline
$A_{6,63}^{-1}$&$(e^{16}+e^{35},-e^{26}+e^{45}+e^{46},e^{36},-e^{46},0,0)$&$-$&$-$\\\hline
$A_{6,70}^{0,0}$&$(-e^{26}+e^{35},e^{16}+e^{45},-e^{46},e^{36},0,0)$&$-$&$-$\\\hline
$A_{6,76}^{-1}$&$(-e^{16}+e^{25},e^{45},e^{24}+e^{36},e^{46},-e^{56},0)$&$-$&$-$\\\hline
$A_{6,78}$&$(-e^{16}+e^{25},e^{45},e^{24}+e^{36}+e^{46},e^{46},$&$-$&$-$\\
&$-e^{56},0)$&&\\\hline
$B_{6,3}^{0}$&$(e^{45},e^{15}+e^{36},e^{14}-e^{26},-e^{56},e^{46},0)$&$-$&$-$\\\hline
$B_{6,4}^1$&$(e^{45},e^{15}+e^{36},e^{14}-e^{26}+e^{56},-e^{56},e^{46},0)$&$(e^1-i\frac{e^6}{2})\wedge(e^2+ ie^3)\wedge(e^4- ie^5)$&$-$\\\hline
$A_{6,82}^{0,1,b}$&$(e^{24}+e^{35},e^{26},be^{36},-e^{46},-be^{56},0)$&$-$&$-$\\
$0\leq b<1$&&&\\\hline
$A_{6,82}^{0,1,1}$&$(e^{24}+e^{35},e^{26},e^{36},-e^{46},-e^{56},0)$&$(e^1-i\frac{e^6}{2})\wedge(e^2-ie^3)\wedge(e^4-ie^5)$&$-$\\\hline
$A_{6,83}^{0,1}$&$(e^{24}+e^{35},e^{26},e^{26}+e^{36},-e^{46}-e^{56},$&$-$&$-$\\
&$-e^{56},0)$&&\\\hline
$A_{6,88}^{0,1,b}$& $(e^{24}+e^{35},e^{26}-be^{36},be^{26}+e^{36},-e^{46}-be^{56},$ &$-$&$-$\\
$b>0$& \ $be^{46}-e^{56},0)$ &&\\\hline
$A_{6,88}^{0,0,1}$&$(e^{24}+e^{35},-e^{36},e^{26},-e^{56},e^{46},0)$&$(e^1+i\frac{e^6}{2})\wedge(e^2+ie^4)\wedge(e^3+ie^5)$&$-$\\\hline
$A_{6,89}^{0,1,b}$&$(e^{24}+e^{35},be^{26},-e^{56},-be^{46},e^{36},0)$&$-$&$-$\\
$b\in \mathbb{R}$&&&\\\hline
$A_{6,90}^{0,\pm 1}$&$(e^{24}+e^{35},e^{46},\pm e^{56},0,\mp e^{36},0)$&$-$&$-$\\\hline
$A_{6,93}^{0,1}$&$(e^{24}+e^{35},-e^{56},-e^{46}-e^{56},e^{26}+e^{36},e^{26},0)$&$-$&$-$\\\hline
$B_{6,6}^{a},\ a\neq0 $&$(e^{24}+e^{35},e^{46},a e^{56},-e^{26},-a e^{36},0)$&$-$&$-$\\
$-1<a<1$&&&\\\hline
$B_{6,6}^{1}$&$(e^{24}+e^{35},e^{46},e^{56},-e^{26},-e^{36},0)$&$(e^1-i\frac{e^6}{2})\wedge(e^2+ie^4)\wedge(e^3-ie^5)$&$-$\\\hline
$A_{6,94}^{-2}$&$(e^{25}+e^{34},-e^{26}+e^{35},-2e^{36},2e^{46},e^{56},0)$&$-$&$-$\\\hline
\end{tabular}
}
}

\medskip

\centerline{{\bf Table 2 (continued).} Indecomposable unimodular (non-nilpotent) solvable}

\centerline{Lie algebras with $b_3\geq 2$.}

\bigskip

\medskip

\noindent {\bf Acknowledgments.}
This work has been partially supported through Project MICINN (Spain) MTM2011-28326-C02-01 and  by GNSAGA (Indam) of Italy.
We thank the referees for useful comments and suggestions that have helped us to improve 
the final version of the paper.

\end{document}